\documentclass[english, a4paper, 11pt]{article}

\usepackage[english]{babel}
\usepackage[T1]{fontenc}
\usepackage{authblk}

\usepackage[a4paper,top=3cm,bottom=3cm,left=3cm,right=3cm,marginparwidth=1.75cm]{geometry}

\usepackage[colorlinks=true, allcolors=blue, hypertexnames=false]{hyperref}
\usepackage{empheq}
\usepackage{graphicx}
\usepackage[colorinlistoftodos]{todonotes}
\usepackage{amsmath}
\usepackage{mathtools}
\usepackage{upgreek}
\usepackage{amssymb}
\usepackage{amsfonts}
\usepackage{amsthm}
\usepackage{hyperref}
\usepackage{enumitem}
\usepackage{mathrsfs}
\usepackage{fontenc}
\usepackage{verbatim}
\usepackage{framed}
\usepackage{algorithm}
\usepackage{bbm}
\usepackage{algpseudocode}
\usepackage{appendix}
\usepackage{color}
\usepackage{multirow}
\usepackage[figuresright]{rotating}

\numberwithin{equation}{section}
\theoremstyle{plain}
\newtheorem{thm}{Theorem}

\newtheorem{prop}[thm]{Proposition}

\newtheorem{assu}{Assumption}
\newtheorem{defn}{Definition}[section]

\theoremstyle{definition}

\theoremstyle{remark}

\newcommand{\transpose}{^{\operatorname{T}}}

\title{A novel control method for solving high-dimensional Hamiltonian systems through deep neural networks}
\author[1]{Shaolin Ji}
\author[2]{Shige Peng}
\author[1]{Ying Peng}
\author[2]{Xichuan Zhang}
\affil[1]{Zhongtai Securities Institute for Financial Studies, Shandong University, 250100, China}
\affil[2]{School of Mathematics, Shandong University, 250100, China}
\begin{document}
\maketitle

\begin{abstract}

In this paper, we mainly focus on solving high-dimensional stochastic Hamiltonian systems with boundary condition,
which is essentially a Forward Backward Stochastic Differential Equation (FBSDE in short), and propose a novel method from the view of the stochastic control. In order to obtain the approximated solution of the Hamiltonian system,
we first introduce a corresponding stochastic optimal control problem such that the extended Hamiltonian system of the control problem is exactly what we need to solve, then we develop two different algorithms suitable for different cases of the control problem and approximate the stochastic control via deep neural networks. From the numerical results, comparing with the Deep FBSDE method developed previously from the view of solving FBSDEs, the novel algorithms converge faster, which means that they require fewer training steps, and demonstrate more stable convergences for different Hamiltonian systems.

\end{abstract}

\section{Introduction}\label{sec:intro}

The theory of the Hamiltonian system is known as one of the dominant tools for the description of dynamic phenomenons in the field of physics and economics~\cite{energy2001}.
For example, in physics, the mechanical and electrical systems are usually represented as energy functions, which are at the same time Hamiltonian systems.
Actually, the Hamiltonian system could reflect the laws of energy conservation and dissipation~\cite{energy2001,Ortega2002}.

A determined Hamiltonian system can be given as
\begin{equation}\label{eq-Det-Hsys}
    \begin{cases}
        \mathrm{d} x_t = H_y(x_t,y_t)\mathrm{d} t , \vspace{1ex} \\
         \mathrm{d} y_t = H_x(x_t,y_t)\mathrm{d} t ,
    \end{cases}
\end{equation}
where $H(x,y):\mathbb{R}^n \times \mathbb{R}^n \rightarrow \mathbb{R}$ is a given real function called the Hamiltonian,
$H_x(\cdot)$ and $H_y(\cdot)$ are partial derivatives of $H(\cdot)$ with respect to $x$ and $y$, respectively.
When considering a terminal condition $y_T = \Phi_x(x_T)$ for a given function $\Phi:\mathbb{R}^n \rightarrow \mathbb{R}$,
\eqref{eq-Det-Hsys} becomes a boundary problem.

For more complex environments where the physical system can not be represented with deterministic form, the Hamiltonian system is usually combined with a stochastic process.  Here we consider a boundary problem of stochastic Hamiltonian system, as shown in the following,
\begin{equation}\label{eq:Ham_sys_sec1}
    \begin{cases}
        \mathrm{d} x_t = H_y(t,x_t,y_t,z_t) \mathrm{d} t + H_z(t,x_t,y_t,z_t) \mathrm{d} B_t, \vspace{1ex} \\
        -\mathrm{d} y_t = H_x(t,x_t,y_t,z_t) \mathrm{d} t - z_t \mathrm{d} B_t, \vspace{1ex} \\
        x_0 = a, \qquad y_T = -\Phi_x(x_T),
    \end{cases}
\end{equation}
which is essentially a fully coupled forward-backward stochastic differential equation (FBSDE in short).
Many research work have studied the solutions of FBSDEs and the eigenvalue of the Hamiltonian systems~\cite{hu1995solution,wu1999Fully,Peng2000eigenvalues,Hu2000FBSDE,ma_solving_1994,ma1993solvability,ma2002approximate,karoui_backward_1997,pardoux_forward-backward_1999}.
The significance of studying this kind of Hamiltonian system is that on the one hand it can be applied in solving the stochastic optimal control problems via the well-known stochastic maximum principle~\cite{bensoussan1982lectures,Yong_stochastic_control}; on the other hand, it helps to obtain the solutions of nonlinear partial differential equations (PDEs in short) according to the connection between the FBSDEs and the PDEs ~\cite{pardoux_forward-backward_1999}.

In most cases, it is difficult to obtain the explicit solution of the Hamiltonian system \eqref{eq:Ham_sys_sec1},  thus numerical methods should be studied. As  \eqref{eq:Ham_sys_sec1} is essentially a FBSDE, an intuitive way is to solve \eqref{eq:Ham_sys_sec1} from the perspective of FBSDEs.
Therefore, numerical methods for solving the FBSDEs can be applied ~\cite{numericalPDE2012review,regression_for_BSDEs,quantization_for_BSDEs,FBSDE_Ma2008,time_discretize_FBSDE_Zhang,Zhao2016Multistep,weinan2017on,Milstein2004Numerical,Yu2016Efficient,huijskens2016efficient,ma1993solvability}
, such as the PDE methods, the probabilistic methods, etc.
However, most of the traditional numerical methods can not deal with high-dimensional problems.
Moreover, it is worth to point out that solving the fully coupled FBSDEs numerically has been a much more challenging problem than the general FBSDEs,
even for low dimensional cases.
Recently, with the application of the deep learning technique in a wide range of areas,
numerical methods based on deep neural networks have been proposed to solve high-dimensional Backward Stochastic Differential Equations (BSDEs in short) and FBSDEs and achieved remarkable success.
Among them, a breakthrough work was developed by ~\cite{WeinanDLforBSDE,WeinanDLforHigh_dim},
the main idea is to reformulate the BSDE into a stochastic optimal control problem by rewriting the backward process into a forward form and taking the terminal error as the cost functional, then the solution of the BSDE is approximated by deep neural network. Other different deep learning algorithms are proposed to solve the BSDEs and related PDEs~\cite{hure_deep_2020,pham_neural_2021,raissi2018forward-backward,wang2018deep}, where they also focus on approximating the solution of the BSDE (or PDE) with the deep neural network. For solving coupled and fully coupled FBSDEs, \cite{deeplearning_FBSDE,Peng_FBSDE_numerical} developed numerical algorithms which are also inspired by ~\cite{WeinanDLforBSDE,WeinanDLforHigh_dim}.

In this paper, we propose a novel method to solve the Hamiltonian system \eqref{eq:Ham_sys_sec1} via deep learning.
As equation \eqref{eq:Ham_sys_sec1} is at the same time a fully coupled FBSDE, this method is also suitable for solving fully coupled FBSDEs. However, different from the above mentioned deep learning methods which aim to solve the FBSDEs directly,
we first look for the corresponding stochastic optimal control problem of the Hamiltonian system,
such that the extended Hamiltonian system of the stochastic control problem is exactly what we need to solve.
Then we approximate the optimal control with deep neural networks.
In order to solve the optimal control problem,
two different cases are considered which correspond to two different algorithms.
The first algorithm (Algorithm \ref{alg:main}) deals with the case where the function $f(t,x,u,v)$ defined in \eqref{eq:f_defn_two_linear} has an explicit form.
For the case that $f(t,x,u,v)$ cannot be expressed explicitly,
the original control problem is transformed to a double objective optimization problem and we develop the second algorithm (Algorithm \ref{alg:double}) to solve it.
Finally, the numerical solutions $(y_t,z_t)$ of \eqref{eq:Ham_sys_sec1} are obtained by calculating the solution of the extended Hamiltonian system for the optimal control according to the stochastic maximum principle.

We also compare the results of our novel proposed algorithms with that of the algorithm developed in our previous work \cite{Peng_FBSDE_numerical} ( called the Deep FBSDE method here),
which can be used to solve the Hamiltonian system from the view of the FBSDEs.
Comparing with the Deep FBSDE method,
the novel algorithms have two advantages.
The first advantage is that less numbers of iteration steps are required to achieve convergent results.
When the Deep FBSDE method converges,
it needs more iterations to achieve a convergent result.
The second advantage is that our proposed algorithms have more stable convergences.
For some Hamiltonian systems,
the Deep FBSDE method is easier to diverge with the same piecewise decay learning rate as our novel proposed algorithms.
The details can be referred to the numerical results in \autoref{sec:num_results}.

This paper is organized as follows.
In \autoref{sec-form}, we describe the Hamiltonian system that we aim to solve, and introduce its corresponding stochastic optimal control problem.
In \autoref{sec:nemrical_method},
we introduce two schemes to solve the stochastic control problem according to whether the function $f(t,x,u,v)$ defined in \eqref{eq:f_defn_two_linear} has an explicit form,
and then give the corresponding neural network architectures.
The numerical results for different examples are shown in \autoref{sec:num_results},
and a brief conclusion is made in \autoref{sec-con}.

\section{Problem formulation}\label{sec-form}

In this section, we first introduce the stochastic Hamiltonian system that we aim to solve, then we show that solving this kind of stochastic Hamiltonian systems is equivalent to solving a stochastic optimal control problem.

\subsection{The stochastic Hamiltonian system}\label{ssec-2-1}

Let $T>0$,
$ (\Omega,\mathcal{F},\mathbb{F},\mathbb{P}) $ be a filtered probability space,
in which $ B:[0,T] \times \Omega \rightarrow \mathbb{R}^d $ is a $ d $-dimensional standard $ \mathbb{F} $-Brownian motion;
$ \mathbb{F}=\{\mathcal{F}_{t}\}_{0\leq t\leq T} $ is the natural filtration generated by the Brownian motion $ B $
Suppose that $ (\Omega,\mathcal{F},\mathbb{P}) $ is complete,
$ \mathcal{F}_{0} $ contains all the $ \mathbb{P} $-null sets in $ \mathcal{F} $ and $ \mathbb{F} $ is right continuous.

For $z^1, z^2\in \mathbb{R}^{n\times d}$,
define $\left\langle z^1, z^2 \right\rangle = \mbox{tr} (z^1(z^2)\transpose{})$ and $|z|^2 = \left\langle z, z \right\rangle$.
The space of all mean square-integrable $\mathcal{F}_t$-adapted and $\mathbb{R}^n$-valued processes will be denoted by $M^2(0,T;\mathbb{R}^n)$, which is a Hilbert space with the norm
\begin{equation*}
    \|v(\cdot)\| = \Big( \mathbb{E} \big[ \int_0^T |v(t)|^2 d t \big] \Big)^{1/2}
\end{equation*}
and
\begin{equation*}
L^2(\Omega, \mathcal{F}_t, \mathbb{P}) \triangleq \left\{ \xi|\xi \in \mathbb{R}^n \mbox{ is } \mathcal{F}_t\mbox{-measurable and } \mathbb{E}\left[|\xi|^2\right] <\infty\right\}.
\end{equation*}
Let
\begin{equation}\label{eq:H_function}
H: [0,T] \times \mathbb{R}^n \times \mathbb{R}^n \times \mathbb{R}^{n \times d} \rightarrow \mathbb{R},
\end{equation}
be a $C^2$ real function of $(x,y,z)$,
called a Hamiltonian and let
\begin{equation}\label{eq:phi_function}
\Phi: \mathbb{R}^n \rightarrow \mathbb{R},
\end{equation}
be a $C^1$ real function of $x$.
In our context, unless otherwise stated, we always assume that the Hamiltonian $H$ is strictly convex with respect to $y$ and $z$.

Consider the following stochastic Hamiltonian system:
\begin{equation}\label{eq:Ham_sys}
    \begin{cases}
        \mathrm{d} x_t = H_y(t,x_t,y_t,z_t) \mathrm{d} t + H_z(t,x_t,y_t,z_t) \mathrm{d} B_t, \vspace{1ex} \\
        -\mathrm{d} y_t = H_x(t,x_t,y_t,z_t) \mathrm{d} t - z_t \mathrm{d} B_t, \vspace{1ex} \\
        x_0 = a, \qquad y_T = -\Phi_x(x_T),
    \end{cases}
\end{equation}
where $H_x$, $H_y$, $H_z$ are derivatives of $H$ with respect to $x$, $y$, $z$, respectively. And the above system is essentially a special kind of fully coupled FBSDEs.

Set
\begin{equation*}
    w =
    \begin{pmatrix}
        x\\
        y\\
        z
    \end{pmatrix} \in \mathbb{R}^n \times \mathbb{R}^n \times \mathbb{R}^{n\times d}, \qquad A(t,w) =
    \begin{pmatrix}
        -H_x \\
        H_y\\
        H_z
    \end{pmatrix}(t,w),
\end{equation*}
and
\begin{equation*}
    \left\langle w^1, w^2 \right\rangle = \left\langle x^1, x^2 \right\rangle + \left\langle y^1, y^2 \right\rangle + \left\langle z^1, z^2 \right\rangle.
\end{equation*}

\begin{defn}
A triple of process $(x(\cdot), y(\cdot), z(\cdot)): [0,T] \times \Omega \rightarrow \mathbb{R}^n \times \mathbb{R}^n \times \mathbb{R}^{n\times d}$ is called an adapted solution of \eqref{eq:Ham_sys},
if $(x(\cdot), y(\cdot), z(\cdot)) \in M^2(0,T;\mathbb{R}^n \times \mathbb{R}^n \times \mathbb{R}^{n\times d})$, and it satisfies \eqref{eq:Ham_sys}.
\end{defn}

\begin{assu}\label{assu:1}
For any $w, w' \in \mathbb{R}^n \times \mathbb{R}^n \times \mathbb{R}^{n\times d}$ and $ x, x' \in \mathbb{R}^n $,
\begin{itemize}
  \item[(i)] there exists a constant $\beta_1>0$, such that
        \begin{equation*}
        \begin{array}{c}
            |A(t,w)-A(t,w')| \leq \beta_1|w-w'|, \vspace{1ex} \\
        \end{array}
        \end{equation*}
        and
        \begin{equation*}
            |\Phi_x(x)-\Phi_x(x')| \leq \beta_1|x-x'|.
        \end{equation*}
  \item[(ii)] there exists a constant $\beta_2>0$,
    such that the following monotonic conditions hold.
        \begin{align*}
            \langle A(t,w) - A(t,w'), w-w' \rangle & \leq -\beta_2 |w-w'|^2\\
            \langle -\Phi_x(x)+\Phi_x(x'), x-x' \rangle & \geq \beta_2 |x-x'|^2.
        \end{align*}
\end{itemize}
\end{assu}

As equation \eqref{eq:Ham_sys} is at the same time a fully coupled FBSDE, we recall the following existence and uniqueness theorem in \cite{hu1995solution, wu1999Fully}.

\begin{thm}[Theorem 3.1 in \cite{hu1995solution}]\label{thm:exist_unique_thm}
    Let Assumption \ref{assu:1} hold.
    Then there exists a unique adapted solution $(x(\cdot),y(\cdot),z(\cdot))$ for \eqref{eq:Ham_sys}.
\end{thm}

Recently, numerical algorithms for solving the BSDEs and FBSDEs with deep learning method ~\cite{WeinanDLforBSDE,WeinanDLforHigh_dim,deeplearning_FBSDE,Peng_FBSDE_numerical} have been proposed and demonstrated remarkable performance. The main idea is to reformulate the BSDE into a stochastic optimal control problem, where the solution $z$ of the BSDE is regarded as a control and approximated with deep neural network, and the terminal error is taken as the cost functional. Other different numerical algorithms have also been developed for solving the FBSDEs and the related PDEs \cite{hure_deep_2020,pham_neural_2021,raissi2018forward-backward,wang2018deep}, which also regard the solution of the FBSDE ($y$ or $z$) as a control and approximate it with appropriate loss function.


\subsection{A novel method to solve the stochastic Hamiltonian system}\label{ssec-2-2}

As noted in the previous sections, the stochastic Hamiltonian system \eqref{eq:Ham_sys} is essentially a fully coupled FBSDE and can be solved through the methods for solving the FBSDEs.
In this paper, we develop a novel method for solving the Hamiltonian system\eqref{eq:Ham_sys}. Different from the above mentioned algorithms for solving the BSDEs or FBSDEs, our main idea is to find the corresponding stochastic optimal control problem of the stochastic Hamiltonian system, and then directly apply the deep learning method to solve the control problem.

$ \forall x,y,u \in \mathbb{R}^n, z,v\in\mathbb{R}^{n\times d} $, set
\begin{equation}\label{eq-F-def}
    F(t,x,u,v,y,z) = \langle y, u\rangle + \langle z,v\rangle - H(t,x,y,z)
\end{equation}
and
\begin{equation}\label{eq:f_defn_two_linear}
    f(t,x,u,v) = \max_{y,z} F(t,x,u,v,y,z).
\end{equation}
Here $f(t,x,u,v)$ is the Legendre-Fenchel transform of $H(t,x,y,z)$ with respect to $(y,z)$. Due to the differentiability and strict convexity of the Hamiltonian $H$,
$f$ is also differential and strictly convex with respect to $u$ and $v$ \cite{convex1970}.

Consider the following control system,
\begin{equation}\label{eq:sys_two_net_linear}
\begin{cases}
    \mathrm{d} x_t = u_t \mathrm{d} t + v_t \mathrm{d} B_t, \vspace{1ex} \\
    x_0 = a,
\end{cases}
\end{equation}
and the cost functional is given as
\begin{equation}\label{eq:cost-two-net}
    J(u(\cdot),v(\cdot)) = \mathbb{E} \left[ \int_0^T f(t,x_t,u_t,v_t) \mathrm{d} t + \Phi(x_T)\right],
\end{equation}
where the controls $ u(\cdot) $ and $ v(\cdot) $ belong to $M^2(0,T;\mathbb{R}^n)$ and $M^2(0,T;\mathbb{R}^{n\times d})$, respectively. The set of all admissible controls is denoted by $ \mathcal{U}_{ad}[0,T] $.
Any $(u^*(\cdot),v^*(\cdot)) \in \mathcal{U}_{ad}[0,T]$ satisfying
\begin{equation}\label{eq:optimal_control_u_bar}
    J(u^*(\cdot),v^*(\cdot)) = \inf_{(u(\cdot),v(\cdot))\in\mathcal{U}_{ad}[0,T]} J(u(\cdot),v(\cdot))
\end{equation}
is called an \textit{optimal control}.
The corresponding state trajectory $x^*(\cdot)$ is called an \textit{optimal trajectory} and the corresponding triple $(x^*(\cdot),u^*(\cdot),v^*(\cdot))$ is called an \textit{optimal triple}.

In the following we prove that solving the stochastic Hamiltonian system \eqref{eq:Ham_sys} is equivalent to solving the stochastic optimal control problem \eqref{eq:sys_two_net_linear}-\eqref{eq:cost-two-net}.

We need the following assumption.

\begin{assu}\label{assu:3}
        $f(t,x,u,v)$ is continuously differentiable with respect to $x$, $u$, $v$, and
        \begin{equation*}
            \begin{aligned}
                |f_x(t,x,u,v)| &\leq C (|x|+|u|+|v|+1), \vspace{1ex} \\
                |f_u(t,x,u,v)| &\leq C (|x|+|u|+|v|+1), \vspace{1ex} \\
                |f_v(t,x,u,v)| &\leq C (|x|+|u|+|v|+1), \vspace{1ex} \\
                |f(t,0,0,0)| &\leq C
            \end{aligned}
        \end{equation*}
        for some given $C>0$.
\end{assu}

\begin{thm}\label{thm-main}
    Let $H$ be a given $C^2$ real function and strictly convex with respect to $y$ and $z$.
    The derivatives of $H$ and $\Phi$ satisfy Assumption \ref{assu:1}; $f$ satisfies Assumption \ref{assu:3}.
    Suppose that  $(x^*(\cdot),u^*(\cdot),v^*(\cdot))$ is the optimal triple of the optimal control problem \eqref{eq:sys_two_net_linear}-\eqref{eq:cost-two-net}.
    Then $(x^*(\cdot),y^*(\cdot),z^*(\cdot))$ uniquely solves the Hamiltonian system \eqref{eq:Ham_sys},
    where $y^*(\cdot)$ can be given as
    \begin{equation}\label{eq-yzF-solution-1}
    \begin{aligned}
        y_t^* & = \mathbb{E} \left[\int_t^T -f_x(s,x_s^*,u^*_s,v^*_s) \mathrm{d} s - \Phi_x(x^*_T) \Big | \mathcal{F}_t \right],
    \end{aligned}
    \end{equation}
    and $z^*(\cdot)$ can be solved by this following BSDE
    \begin{equation}\label{eq-yzF-solution}
            \left\{
            \begin{array}{l}
            -\mathrm{d} y_t^* = -f_x(t,x^*_t,u^*_t,v^*_t) \mathrm{d} t - z_t^*\mathrm{d} B_t, \vspace{1ex}\\
            y_T^*=-\Phi_x(x^*_T), \qquad t\in[0,T].
            \end{array}
            \right.
    \end{equation}
\end{thm}

\begin{proof}
     Set
    \begin{equation}\label{eq:h_defn_two_linear}
        h(t,x,u,v,y,z) = \langle y, u \rangle + \langle z, v \rangle - f(t,x,u,v).
    \end{equation}

    Under our assumptions,
    for any optimal triple $(x^*(\cdot),u^*(\cdot),v^*(\cdot))$ of the optimal control problem \eqref{eq:sys_two_net_linear}-\eqref{eq:cost-two-net},
    we have the following extended stochastic Hamiltonian system through the well-known stochastic maximum principle (SMP in short) ( e.g. Theorem 4.1 in \cite{bensoussan1982lectures}),
    \begin{equation}\label{eq-ext-Hsys}
    \begin{cases}
    \mathrm{d} x_t^* = u_t^* \mathrm{d} t + v_t^* \mathrm{d} B_t, \vspace{1ex} \\
    -\mathrm{d} y_t^* = h_x(t,x^*_t,u^*_t,v^*_t,y^*_t,z^*_t) \mathrm{d} t - z_t^*\mathrm{d} B_t, \vspace{1ex}\\
    x_0^* = a, \qquad y_T^*=-\Phi_x(x^*_T), \qquad t\in[0,T],
    \end{cases}
    \end{equation}
    and
    \begin{equation}\label{eq:max_condition_h}
        h(t,x^*_t,u^*_t,v^*_t,y^*_t,z^*_t) = \max_{u\in\mathbb{R}^n, \\ v\in\mathbb{R}^{n\times d}} h(t,x^*_t,u,v,y^*_t,z^*_t), \ a.e.\ t\in[0,T].
    \end{equation}
    The solution of the extended stochastic Hamiltonian system \eqref{eq-ext-Hsys}-\eqref{eq:max_condition_h} is a 5-tuple $(x^*(\cdot),u^*(\cdot),v^*(\cdot),y^*(\cdot),z^*(\cdot))$.

    By Theorem 12.2 in \cite{convex1970},
    we have the inverse Legendre-Fenchel transform of \eqref{eq:f_defn_two_linear}:
    \begin{equation}\label{eq-H-uv}
        H(t,x,y,z) = \max_{u,v} h(t,x,u,v,y,z).
    \end{equation}
    Because $h(t,x,u,v,y,z)$ is strictly concave in $u,v$, it yields that the maximum point $(u^*_t,v^*_t)$ of \eqref{eq:max_condition_h} is uniquely determined by $(x^*_t,y^*_t,z^*_t)$ due to the implicit function existence theorem:
    \begin{equation}
        \begin{aligned}
        u^*_t &= \bar{u}(t,x^*_t,y^*_t,z^*_t), \\
        v^*_t &= \bar{v}(t,x^*_t,y^*_t,z^*_t),
        \end{aligned}
    \end{equation}
   and $\bar{u}, \bar{v}$ are differentiable functions.
    By \eqref{eq-H-uv},
    \begin{equation}
        \begin{aligned}
            H(t,x^*_t,y^*_t,z^*_t) &= h(t,x^*_t,\bar{u}(t,x^*_t,y^*_t,z^*_t),\bar{v}(t,x^*_t,y^*_t,z^*_t),y^*_t,z^*_t),
        \end{aligned}
    \end{equation}
    which leads to
    \begin{equation}
    \begin{aligned}
        &f(t,x^*_t,\bar{u}(t,x^*_t,y^*_t,z^*_t),\bar{v}(t,x^*_t,y^*_t,z^*_t)) \\
        = &\langle y^*_t, \bar{u}(t,x^*_t,y^*_t,z^*_t) \rangle + \langle z^*_t, \bar{v}(t,x^*_t,y^*_t,z^*_t) \rangle - H(t,x^*_t,y^*_t,z^*_t).
    \end{aligned}
    \end{equation}
    Thus, the derivatives of $f(t,x,u,v)$ with respect to $u,v$ are
    \begin{equation}
        \begin{aligned}
        f_u(t,x^*_t,\bar{u}(t,x^*_t,y^*_t,z^*_t),\bar{v}(t,x^*_t,y^*_t,z^*_t)) = y^*_t, \\
        f_v(t,x^*_t,\bar{u}(t,x^*_t,y^*_t,z^*_t),\bar{v}(t,x^*_t,y^*_t,z^*_t)) = z^*_t.
        \end{aligned}
    \end{equation}
   It can be verified that
    \begin{equation}
        \begin{aligned}
            H_x(t,x^*_t,y^*_t,z^*_t) =& -f_x(t,x^*_t,\bar{u}(t,x^*_t,y^*_t,z^*_t),\bar{v}(t,x^*_t,y^*_t,z^*_t))\\
            =& -f_x(t,x^*_t,u^*_t,v^*_t)\\
            = &h_x(t,x^*_t,u^*_t,v^*_t,y^*_t,z^*_t),
        \end{aligned}
    \end{equation}
    which implies that $(y^*(\cdot),z^*(\cdot))$ solves the BSDE \eqref{eq-yzF-solution}.
    Taking the conditional expectation in the backward equation of \eqref{eq-yzF-solution},
    we have \eqref{eq-yzF-solution-1} hold.

    Similarly, we have
    \begin{equation}
        \begin{aligned}
            H_y(t,x^*_t,y^*_t,z^*_t) = \bar{u}(t,x^*_t,y^*_t,z^*_t)=u^*_t, \\
            H_z(t,x^*_t,y^*_t,z^*_t) = \bar{v}(t,x^*_t,y^*_t,z^*_t)=v^*_t.
        \end{aligned}
    \end{equation}
    Thus, the extended stochastic Hamiltonian system \eqref{eq-ext-Hsys} is just the Hamiltonian system \eqref{eq:Ham_sys} and
     $(x^*(\cdot),y^*(\cdot),z^*(\cdot))$ solves \eqref{eq:Ham_sys}.  Because $H$ satisfies the monotonicity condition in Assumption \ref{assu:1}, the uniqueness of the solution is proved.
\end{proof}

The following proposition can help us to construct our algorithms in the next section.
\begin{prop}\label{cor-1}
    Under the same assumptions as in Theorem \ref{thm-main}, we have
    \begin{equation}\label{eq-J-solution-F}
    \begin{aligned}
        J(u^*(\cdot), v^*(\cdot)) &= \mathbb{E} \left[\int_0^T f(t,x^*_t,u^*_t,v^*_t) \mathrm{d}t + \Phi(x_T^*) \right]\\
        &= \mathbb{E} \left[\int_0^T F(t,x^*_t,u^*_t,v^*_t,y^*_t,z^*_t) dt + \Phi(x_T^*) \right],
    \end{aligned}
    \end{equation}
    and $(y^*(\cdot),z^*(\cdot))$ of \eqref{eq:Ham_sys} can also be obtained by solving the following BSDE:
    \begin{equation}\label{eq-yzF-solution-F}
    \left\{
    \begin{array}{l}
    -\mathrm{d} y_t^* = -F_x(t,x^*_t,u^*_t,v^*_t,y^*_t,z^*_t) \mathrm{d} t - z_t^*\mathrm{d} B_t, \vspace{1ex}\\
    y_T^*=-\Phi_x(x^*_T), \qquad t\in[0,T],
    \end{array}
    \right.
    \end{equation}
    where
    \begin{equation}
        F(t,x^*_t,u^*_t,v^*_t,y^*_t,z^*_t) = \max_{y,z} F(t,x^*_t,u^*_t,v^*_t,y,z).
    \end{equation}
    Then $y^*(\cdot)$ can be expressed as
    \begin{equation}\label{eq-y_F}
        \begin{aligned}
            y_t^* & = \mathbb{E} \left[\int_t^T -F_x(s,x_s^*,u^*_s,v^*_s,y^*_s,z^*_s) \mathrm{d} s - \Phi_x(x^*_T) \Big | \mathcal{F}_t \right].
        \end{aligned}
    \end{equation}
\end{prop}

\begin{proof}
    By the definition of $H$ and the SMP,
    \begin{equation}
    \begin{aligned}
        H(t,x^*_t,y^*_t,z^*_t) &= \max_{u,v} h(t,x^*_t,u,v,y^*_t,z^*_t), \\
        &=\langle y^*_t, u^*_t \rangle + \langle z^*_t, v^*_t \rangle - f(t,x^*_t,u^*_t,v^*_t).
    \end{aligned}
    \end{equation}
   Then, we have
    \begin{equation}
        f(t,x^*_t,u^*_t,v^*_t) = \langle y^*_t, u^*_t \rangle + \langle z^*_t, v^*_t \rangle - H(t,x^*_t,y^*_t,z^*_t).
    \end{equation}
    The definition of $f(t,x,u,v)$ shows that
    \begin{equation}\label{eq-yz-bar}
        \begin{aligned}
            f(t,x^*_t,u^*_t,v^*_t) &= \max_{y,z} F(t,x^*_t,u^*_t,v^*_t,y,z), \\
            &= \langle \bar{y}, u^*_t\rangle + \langle \bar{z},v^*_t\rangle - H(t,x^*_t,\bar{y},\bar{z}),
        \end{aligned}
    \end{equation}
    where $(\bar{y},\bar{z})$ is the minimum point.

    Notice that $F(t,x,u,v,y,z)$ is strictly concave with respect to $y,z$. If $(y^*_t,z^*_t) \neq (\bar{y},\bar{z})$,
    then
    \begin{equation}
        \begin{aligned}
            F(t,x^*_t,u^*_t,v^*_t,\bar{y},\bar{z}) & > F(t,x^*_t,u^*_t,v^*_t,y^*_t,z^*_t) \\
            & = \langle y^*_t, u^*_t \rangle + \langle z^*_t, v^*_t \rangle - H(t,x^*_t,y^*_t,z^*_t) \\
            & = f(t,x^*_t,u^*_t,v^*_t),
        \end{aligned}
    \end{equation}
    which contradicts the formula \eqref{eq-yz-bar}. So we have $(y^*_t,z^*_t) = (\bar{y},\bar{z})$
    and \eqref{eq-J-solution-F} holds.

    Because the strict concavity of $F(t,x,u,v,y,z)$ with respect to $y,z$,
    we have
    \begin{equation}
        \begin{aligned}
            F_y(t,x^*_t,u^*_t,v^*_t,y^*_t,z^*_t) & = u^*_t - H_y(t,x^*_t,y^*_t,z^*_t) = 0 \\
            F_z(t,x^*_t,u^*_t,v^*_t,y^*_t,z^*_t) & = v^*_t - H_z(t,x^*_t,y^*_t,z^*_t) = 0.
        \end{aligned}
    \end{equation}
    Similar to the proof of Theorem \ref{thm-main},
    according to the implicit function existence theorem, we can easily check that
    \begin{equation}
            F_x(t,x^*_t,u^*_t,v^*_t,y^*_t,z^*_t) = -H_x(t,x^*_t,y^*_t,z^*_t),
    \end{equation}
    holds,
    and then \eqref{eq-yzF-solution-F} holds.
    Taking the conditional expectation on \eqref{eq-yzF-solution-F},
    we have \eqref{eq-y_F}.
\end{proof}


In Theorem \ref{thm-main} and Proposition \ref{cor-1},
we choose the stochastic optimal control problem \eqref{eq:sys_two_net_linear}-\eqref{eq:cost-two-net},
whose diffusion term $b$ and drift term $\sigma$ of the state equation are simplely $u$ and $v$.
In fact, to simplify the linear terms of $y$ and $z$ in the Hamiltonian $H$, we can also choose other forms of $b$ and $\sigma$,
such as $\alpha_1(x)+\alpha_2(x)u$ and $\beta_1(x)+\beta_2(x)v$,
which are linear with respect to $u$ and $v$.
In this cases,
the transformations \eqref{eq:f_defn_two_linear} and \eqref{eq-H-uv} also hold.
We show an example of this form in subsection \ref{ssec-ex3} for the numerical results.

Besides,
we can still solve the Hamiltonian system \eqref{eq:Ham_sys} even if the coefficients $H_x,H_y,H_z$ do not satisfy the monotonic conditions (Assumption \ref{assu:1} (ii)) in Theorem \ref{thm-main} and Proposition \ref{cor-1}.
For example,
the articles \cite{ma1993solvability,ma2002approximate} studied the solvability of FBSDEs under relatively loose conditions.
In this situation,
as long as the optimal controls reach the optimal values,
the Hamiltonian system \eqref{eq:Ham_sys} can be solved,
however the solution is not necessarily unique.

\section{Numerical method for solving Hamiltonian systems}\label{sec:nemrical_method}

In Section \ref{sec-form}, we present the idea of the stochastic optimal control method to solve the Hamiltonian system \eqref{eq:Ham_sys}. According to Theorem \ref{thm-main}, we only need to find the optimal control triple $(x^*(\cdot),u^*(\cdot),v^*(\cdot))$ of the stochastic control problem \eqref{eq:sys_two_net_linear} and \eqref{eq:cost-two-net}, then the solution $(y^*(\cdot),z^*(\cdot))$ can be obtained by taking the conditional expectation on the backward SDE of \eqref{eq-yzF-solution}. Therefore, effective approximation method should be used to obtain the optimal triple of \eqref{eq:sys_two_net_linear}-\eqref{eq:cost-two-net}, especially for high dimensional cases.

Deep neural networks are usually used to approximate functions defined on finite-dimensional space,
and the approximation relies on the composition of layers with simple functions.
On the basis of the universal approximation theorem~\cite{universial_cybenko_1989,multilayer_Appr_1991},
the neural networks have shown to be an effective tool and gained great successes in many practical applications.
In this paper, inspired by \cite{han2016deep},
we simulate the stochastic optimal control problem \eqref{eq:sys_two_net_linear}-\eqref{eq:cost-two-net} from a direct way with the deep neural network and develop two different numerical algorithms suitable for different cases.

Let $\pi$ be a partition of the time interval,
$ 0 = t_{0}<t_{1}<t_{2}<\cdots<t_{N-1}<t_{N} =T $ of $ [0,T] $.
Define $ \Delta t_{i}=t_{i+1}-t_{i} $ and $ \Delta B_{t_i}=B_{t_{i+1}}-B_{t_{i}} $,
where $ B_{t_{i}} \sim \mathcal{N}(0,t_{i}) $, for $ i = 0, 1, 2,\cdots, N-1 $.
We also denote
\begin{equation*}
    \delta = \sup\limits_{0\leq i\leq N-1}\Delta t_i,
\end{equation*}
which is small enough.
Then the Euler-Maruyama scheme of the state equation \eqref{eq:sys_two_net_linear} can be written as
\begin{equation}\label{eq:discrete_scheme_two_linear}
    \left\{
    \begin{array}{l}
        x_{t_{i+1}}^{\pi} = x_{t_i}^{\pi} + u_{t_i}^{\pi} \Delta t_i +  v_{t_i}^{\pi} \Delta B_{t_i}, \vspace{1ex} \\
        x_0 = a,
    \end{array}
    \right.
\end{equation}
and the corresponding cost functional is given as
\begin{equation}\label{eq:discrete_cost}
    J(u^{\pi}(\cdot),v^{\pi}(\cdot)) = \dfrac{1}{M} \sum_{m=1}^M \Big[ \sum_{i=0}^{N-1} f(t_i, x_{t_i}^{\pi,m}, u_{t_i}^{\pi,m}, v_{t_i}^{\pi,m}) \Delta t_i + \Phi(x_{t_N}^{\pi,m}) \Big],
\end{equation}
where $M$ represents the number of Monte Carlo samples.

We introduce a feedforward neural network $\varphi^{\theta}:[0,T]\times\mathbb{R}^{n}\rightarrow\mathbb{R}^{n}$ of the form
\begin{equation}\label{eq:base_net}
    \varphi^{\theta} = \mathcal{A}_{\ell} \circ \sigma_{\ell -1} \circ \mathcal{A}_{\ell -1} \circ \cdots \circ \sigma_{1} \circ \mathcal{A}_1,
\end{equation}
where
\begin{itemize}
    \item $\ell$ is a positive integer specifying the depth of the neural network,
    \item $\mathcal{A}_1,\cdots,\mathcal{A}_{\ell}$ are functions of the form
    \begin{equation*}
        \begin{aligned}
            \mathcal{A}_1 &= w_1 x + b_1 \in \mathbb{R}^{d_{1}}, \\
            \mathcal{A}_i &= w_i \mathcal{A}_{i-1} + b_i \in \mathbb{R}^{d_{i}}, \qquad \text{for } 2\leq i\leq \ell,
        \end{aligned}
    \end{equation*}
    the matrix weights $w_i$ and bias vector $b_i$ are trainable parameters,
    $\theta$ is the whole trainable parameters $\theta=(w_i,b_i)_{1\leq i\leq\ell}$,
    and $d_i$ is the number of neurons at layer $i$,
    \item $\sigma_{\ell-1},\cdots,\sigma_1$ are the nonlinear activation functions, such as the sigmoid, the rectified linear unit (ReLU), the exponential linear unit (ELU), etc.
\end{itemize}

We approximate the controls $u,v$ with two different neural networks,
which can be represented with \eqref{eq:base_net} and denoted as $\varphi^{\theta^u}_u$ and $\varphi^{\theta^v}_v$, respectively:
\begin{equation}\label{eq:control_net_represent}
\begin{cases}
    u = \varphi^{\theta^u}_u(t,x)=\varphi_u(t,x;\theta^u)=\mathcal{A}_{\ell_u}^u \circ \sigma_{\ell_u -1}^u \circ \mathcal{A}_{\ell_u -1}^u \circ \cdots \circ \sigma_{1}^u \circ \mathcal{A}_1^u(t,x) \vspace{1ex} \\
    v = \varphi^{\theta^v}_v(t,x)=\varphi_v(t,x;\theta^v)=\mathcal{A}_{\ell_v}^v \circ \sigma_{\ell_v -1}^v \circ \mathcal{A}_{\ell_v -1}^v \circ \cdots \circ \sigma_{1}^v \circ \mathcal{A}_1^v(t,x).
\end{cases}
\end{equation}
The two neural networks have the same input dimension but different output dimensions. In this paper, we use the common parameters of the neural networks for all the time points, i.e. a single network is developed for simulating each of the control, and the time point $t$ is
regarded as an input of the neural network.

\subsection{Case 1: the function \texorpdfstring{$f(t,x,u,v)$}{} has an explicit form}

When the function $f(t,x,u,v)$ defined as \eqref{eq:f_defn_two_linear} has an explicit form,
then the discrete cost functional \eqref{eq:discrete_cost} can be approximated directly with
\begin{equation}\label{eq:discrete_cost_theta}
    J(u^{\pi}(\cdot),v^{\pi}(\cdot)) = \dfrac{1}{M} \sum_{m=1}^M \Big[ \sum_{i=0}^{N-1} f(t_i, x_{t_i}^{\pi,m}, u_{t_i}^{\pi,m}, v_{t_i}^{\pi,m}) \Delta t_i + \Phi(x_{t_N}^{\pi,m}) \Big],
\end{equation}
which is also the loss function we need to minimize in the whole neural network, and $u_{t_i}^{\pi}, v_{t_i}^{\pi}$ are the outputs of the two neural networks at time $t_i$. Both the neural networks approximating $u_{t_i}^{\pi}, v_{t_i}^{\pi}$ contain one $(n+1)$-dim input layer,
three $(n+10)$-dim hidden layers.
The network of $u_{t_i}^{\pi}$ has an $n$-dim output layer and that of $v_{t_i}^{\pi}$ has a $n\times d$-dim output layer.
In order to simplify the representation, here we use $\theta$ to represent the training parameters $(\theta^u,\theta^v)$ for both of the neural networks.

To minimize the loss function \eqref{eq:discrete_cost_theta} and learn the optimal parameters,
some basic optimization algorithms, such as stochastic gradient descent (SGD), AdaGrad, RMSProp, and Adam which are already implemented in TensorFlow can be used. In this paper, the Adam method~\cite{adam_2017} is adopted as the optimizer.

Once we obtain the approximations of the optimal controls $u^*$ and $v^*$,
we can calculate the numerical solution $y^*_t$ by taking the conditional expectation on the Backward SDE of \eqref{eq-yzF-solution}, which can be approximated with Monte Carlo simulation:
\begin{equation}\label{eq:Y_formulation_para}
    y_{t_i}^{\pi} = \dfrac{1}{M} \sum\limits_{m=1}^M \Big[ \sum\limits_{j=i}^{N-1} -f_x(t_j, x_{t_j}^{\pi,m}, u_{t_j}^{\pi,m}, v_{t_j}^{\pi,m}) \Delta t_j - \Phi_x(x_{t_N}^{\pi,m}) \Big].
\end{equation}

We show the whole network architecture in Figure \ref{fig:Network_Algo1}, where $h^u$ and $h^v$ represent respectively the hidden layers of the neural networks $\varphi^{\theta^u}_u$ and $\varphi^{\theta^v}_v$. For each of the neural network, the common parameters is used for all the time points, and the time $t$ is taken as one of the inputs of the neural network.

\begin{figure}[htbp]
  \centering
  \includegraphics[scale=0.6]{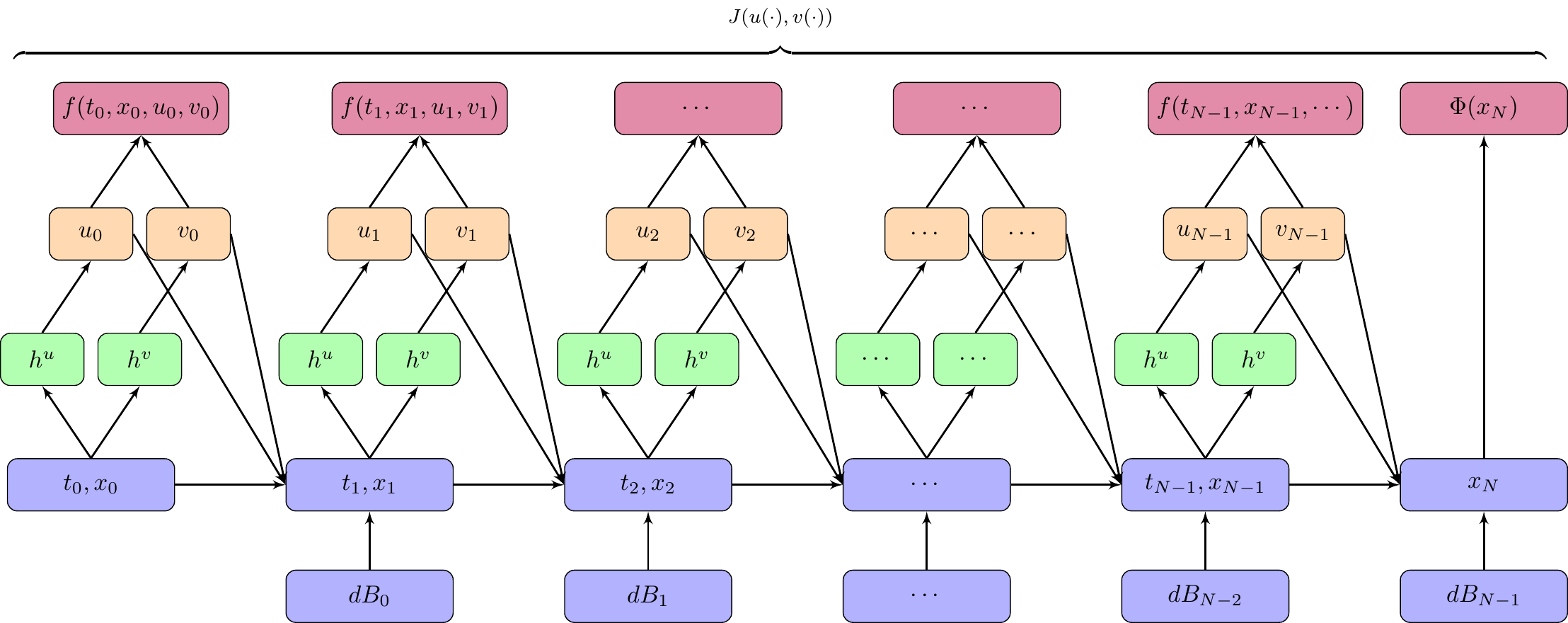}
  \caption{The whole network architecture for case 1: $f(t,x,u,v)$ has an explicit representation. The boxes in purple, green and orange represent respectively the input layer, the hidden layers and the output layers of the neural networks $\varphi^{\theta^u}_u$ and $\varphi^{\theta^v}_v$. The data flow of the neural networks is represented with black arrows.}
  \label{fig:Network_Algo1}
\end{figure}

The pseudo-code is shown in Algorithm \ref{alg:main}.
\begin{algorithm}[H]
  \renewcommand{\thealgorithm}{1}
  \caption{Numerical algorithm for solving the Hamiltonian system of case 1}
  \label{alg:main}
  \begin{algorithmic}[1]
    \Require The Brownian motion $ \Delta B_{t_i} $, initial state $ a $, and time $ t_i $;
    \Ensure The output controls $u_{t_i}^{\pi}$, $v_{t_i}^{\pi}$ and $y_{t_i}^{\pi}$.
    \For { $ l = 0 $ to $ maxstep $}
    \State $ x_{0}^{l,\pi,m} = a $, $ loss = 0 $;
    \For { $ i = 0 $ to $ N-1 $}
    \State $u^{l,\pi,m}_{t_i} = \varphi_u(t_i, x^{l,\pi,m}_{t_i};\theta^l);$
    \State $v^{l,\pi,m}_{t_i} = \varphi_v(t_i, x^{l,\pi,m}_{t_i};\theta^l);$
    \State $x^{l,\pi,m}_{t_{i+1}} = x^{l,\pi,m}_{t_i} + u^{l,\pi,m}_{t_i}\Delta t_i + v^{l,\pi,m}_{t_i} \Delta B_{t_i};$
    \State $loss = loss + f(t_i, x_{t_i}^{l,\pi,m}, u_{t_i}^{l,\pi,m}, v_{t_i}^{l,\pi,m}) \Delta t_i;$
    \EndFor
    \State $loss = \dfrac{1}{M} \sum\limits_{m=1}^M \left[loss + \Phi(x_{t_N}^{l,\pi,m});\right] $
    \State $ \theta^{l+1}= Adam(\theta^{l}, \nabla loss); $
    \EndFor
    \State $y_{t_i}^{\pi} = \dfrac{1}{M} \sum\limits_{m=1}^M \Big[ \sum\limits_{j=i}^{N-1} -f_x(t_j, x_{t_j}^{l,\pi,m}, u_{t_j}^{l,\pi,m}, v_{t_j}^{l,\pi,m}) \Delta t_j - \Phi_x(x_{t_N}^{l,\pi,m}) \Big]. $
  \end{algorithmic}
\end{algorithm}

\subsection{Case 2: the function \texorpdfstring{$f(t,x,u,v)$}{} does not have an explicit form}

For the cases where the function $f$ does not have an explicit form, we can still solve the Hamiltonian system \eqref{eq:Ham_sys} by constructing a different neural network architecture.
For any given optimal triple $(x^*(\cdot),u^*(\cdot),v^*(\cdot))$ of the optimal control problem \eqref{eq:cost-two-net},
we assume that the solution $(y^*(\cdot),z^*(\cdot))$ of \eqref{eq-yzF-solution-F} satisfy
\begin{equation}
    y^*_t = Y(t,x^*_t), \qquad z^*_t = Z(t,x^*_t), \qquad \forall t\in[0,T],
\end{equation}
for some functions $Y$ and $Z$.
Then according to Proposition \ref{cor-1}, we have
\begin{equation}\label{eq-F-sys}
    J(u^*(\cdot), v^*(\cdot)) = \mathbb{E} \displaystyle \left[\int_0^T F(t,x^*_t,u^*_t,v^*_t,y^*_t,z^*_t) dt + \Phi(x_T^*) \right],
\end{equation}
and
\begin{equation}
    \mathrm{d} x_t^* = u^*_t \mathrm{d} t + v_t^*\mathrm{d} B_t, \quad x_0^* = a,
\end{equation}
where $y^*_t,z^*_t$ satisfy
\begin{equation}\label{eq-F-max}
    F(t,x^*_t,u^*_t,v^*_t,y^*_t,z^*_t) = \max\limits_{y,z} F(t,x^*_t,u^*_t,v^*_t,y,z),
\end{equation}
and $F$ is defined by \eqref{eq-F-def}.
Because of the strict concavity and differentiable properties of $F$ with respect to $y,z$,
the constraint condition \eqref{eq-F-max} can be rewritten as
\begin{equation}\label{eq-Fy-zero}
    \begin{cases}
        F_y(t,x^*_t,u^*_t,v^*_t,y^*_t,z^*_t) = 0, \vspace{1ex} \\
        F_z(t,x^*_t,u^*_t,v^*_t,y^*_t,z^*_t) = 0.
    \end{cases}
\end{equation}
In this way,
the Hamiltonian system \eqref{eq:Ham_sys} can be solved by solving the stochastic optimal control
\begin{equation}\label{eq-Fyz-sys}
    J(u(\cdot), v(\cdot)) = \mathbb{E} \displaystyle \left[\int_0^T F(t,x_t,u_t,v_t,y_t,z_t) \mathrm{d}t + \Phi(x_T) \right],
\end{equation}
with the state constraint
\begin{equation}\label{eq-Fyz-state}
\begin{cases}
    \mathrm{d} x_t = u_t \mathrm{d} t + v_t\mathrm{d} B_t, \quad x_0 = a, \vspace{1ex} \\
    F_y(t,x_t,u_t,v_t,y_t,z_t) = 0, \vspace{1ex} \\
    F_z(t,x_t,u_t,v_t,y_t,z_t) = 0.
\end{cases}
\end{equation}

Now we focus on solving \eqref{eq-Fyz-sys}-\eqref{eq-Fyz-state} with a new neural network architecture. Firstly, the Euler-Maruyama scheme \eqref{eq:discrete_scheme_two_linear} is used to obtain the discrete form of the optimal control problem. In addition to the neural networks for simulating the controls $u$ and $v$, we need to construct two more neural networks for simulating the functions $Y$ and $Z$,
\begin{equation}\label{eq-YZ-con}
    \begin{cases}
    y = \varphi^{\theta^y}_y(t,x)=\varphi_y(t,x;\theta^y)=\mathcal{A}_{\ell_y}^y \circ \sigma_{\ell_y -1}^y \circ \mathcal{A}_{\ell_y -1}^y \circ \cdots \circ \sigma_{1}^y \circ \mathcal{A}_1^y(t,x) \vspace{1ex} \\
    z = \varphi^{\theta^z}_z(t,x)=\varphi_z(t,x;\theta^z)=\mathcal{A}_{\ell_z}^z \circ \sigma_{\ell_z -1}^z \circ \mathcal{A}_{\ell_z -1}^z \circ \cdots \circ \sigma_{1}^z \circ \mathcal{A}_1^z(t,x).
    \end{cases}
\end{equation}
We also use the common parameters on all the time points for each of the four neural networks, and the inputs of each neural network are $(t,x)$.
All of the four neural networks contain one $(n+1)$-dim input layer and three $(n+10)$-dim hidden layers. The dimensions of the output layer for each neural network are different, that of $y$ and $u$ are $n$-dim, and that of $z$ and $v$ are $(n\times d)$-dim. We still adopt Adam as the optimizer.

We denote $\theta = (\theta_{uv}, \theta_{yz})$ as all the parameters of the neural networks,
where $\theta_{uv}$ are the parameters of the neural networks $\varphi_u$ and $\varphi_v$ (for simulating $u$ and $v$),
and $\theta_{yz}$ are the parameters of the neural networks $\varphi_y$ and $\varphi_z$ (for simulating $y$ and $z$).

Then the cost functional of the control problem is approximated by
\begin{equation}\label{eq-loss-1}
    J(u^{\pi}(\cdot),v^{\pi}(\cdot)) = \dfrac{1}{M} \sum_{m=1}^M \Big[ \sum_{i=0}^{N-1} F(t_i, x_{t_i}^{\pi,m}, u_{t_i}^{\pi,m}, v_{t_i}^{\pi,m}, y_{t_i}^{\pi,m}, z_{t_i}^{\pi,m}) \Delta t_i + \Phi(x_{t_N}^{\pi,m}) \Big],
\end{equation}
which is the first loss function we need to minimize,
and $u_{t_i}^{\pi}, v_{t_i}^{\pi}, y_{t_i}^{\pi}, z_{t_i}^{\pi}$ are the outputs of the whole neural networks at time $t_i$.
In addition, in order to guarantee that the conditions \eqref{eq-Fy-zero} hold,
we introduce the other cost functional
\begin{equation}\label{eq-loss-2}
\begin{aligned}
    J(y^{\pi}(\cdot),z^{\pi}(\cdot)) &:= \dfrac{1}{M} \sum_{m=1}^M \sum_{i=0}^{N-1} \Big[ | F_y(t_i, x_{t_i}^{\pi,m}, u_{t_i}^{\pi,m}, v_{t_i}^{\pi,m}, y_{t_i}^{\pi,m}, z_{t_i}^{\pi,m}) |^2  \\
    & \qquad \qquad + | F_z(t_i, x_{t_i}^{\pi,m}, u_{t_i}^{\pi,m}, v_{t_i}^{\pi,m}, y_{t_i}^{\pi,m}, z_{t_i}^{\pi,m}) |^2 \Big].
\end{aligned}
\end{equation}
which is at the same time the second loss function we need to minimize in the neural networks.

The update of the neural network parameters is carried out as follows. Suppose that we have finished the update at the iteration step $l$ and obtain the parameters $\theta^l = (\theta_{uv}^l, \theta_{yz}^l)$.
We first calculate the values $(x^{\pi}_{t_i},u^{\pi}_{t_i},v^{\pi}_{t_i},y^{\pi}_{t_i},z^{\pi}_{t_i})$ with the parameters $\theta^l$.
Then the parameters $\theta_{uv}^l$ are updated to $\theta_{uv}^{l+1}$ by using one step Adam optimization with the first loss function \eqref{eq-loss-1}.
In the following,
the parameters $(\theta_{uv}^{l+1}, \theta_{yz}^l)$ are used to calculate the values $(x^{\pi}_{t_i},u^{\pi}_{t_i},v^{\pi}_{t_i},y^{\pi}_{t_i},z^{\pi}_{t_i})$ in \eqref{eq-loss-2}.
Then the parameters $\theta_{yz}^l$ are updated with the second loss function \eqref{eq-loss-2} by Adam optimization.
In each iteration step, the update of parameters $\theta_{yz}^l$ can be performed multiple times,
for example $\kappa$ times,
to ensure the loss function \eqref{eq-loss-2} is enough small.
And after $\kappa$ times update,
the parameters  of the neural networks are denoted as $\theta_{yz}^{l+1}$. Finally the solution $y$ is obtained by taking the conditional expectation of the backward SDE of \eqref{eq-yzF-solution-F} which can be calculated with Monte Carlo simulation:

\begin{equation}\label{eq:Y_F}
    y_{t_i}^{\pi} = \dfrac{1}{M} \sum\limits_{m=1}^M \Big[ \sum\limits_{j=i}^{N-1} -F_x(t_j, x_{t_j}^{\pi,m}, u_{t_j}^{\pi,m}, v_{t_j}^{\pi,m}, y_{t_j}^{\pi,m}, z_{t_j}^{\pi,m}) \Delta t_j - \Phi_x(x_{t_N}^{\pi,m}) \Big].
\end{equation}
 The pseudo-code is given in Algorithm \ref{alg:double}.

In fact,we can also deal with the maximum condition \eqref{eq-F-max} directly. in this situation, we need to maximize the second objective functional in addition to \eqref{eq-loss-1},
\begin{equation}\label{eq-J-F}
    J(y(\cdot),z(\cdot)) = \mathbb{E} \left[\int_0^T F(t,x_t,u_t,v_t,y_t,z_t) \right] \mathrm{d}t.
\end{equation}
and a similar scheme can be given.
The advantage for using the conditions \eqref{eq-Fy-zero} instead of \eqref{eq-J-F} is that we can determine the influence of the constraint conditions by the value of the cost functional \eqref{eq-loss-2}, as the optimal value of \eqref{eq-loss-2} should be 0.

\section{Numerical results}\label{sec:num_results}

In this section,
we show the numerical results for solving the Hamiltonian system with our proposed algorithms.
If not specifically mentioned,
we use 6-layer fully connected neural networks for the approximation in these examples, the number of time divisions is set to be $N=25$,  and we mainly use a piecewise constant decay learning rate which decreases from $3\times 10^{-3}$ to $1\times10^{-3}$ with the increase of the number of iteration steps.
We adopt ELU as the activation function. In order to show the performance of the proposed algorithms,
we compare the results among the two proposed algorithms and the Deep FBSDE method (briefly noted as DFBSDE in the figures and tables of this section) which was developed as Algorithm 1 in our previous work \cite{Peng_FBSDE_numerical}.
But different from \cite{Peng_FBSDE_numerical}, for better comparison with the novel proposed algorithms, here we use the ELU activation function and remove the batch normalization layer in the Deep FBSDE method. For each algorithm of the examples, we perform ten independent runs to show more accurate results.

\begin{algorithm}[H] 
  \renewcommand{\thealgorithm}{2}
  \caption{Numerical algorithm for solving the Hamiltonian system of case 2}
  \label{alg:double}
  \begin{algorithmic}[1]
    \Require The Brownian motion $ \Delta B_{t_i} $, initial state $ a $, and time $ t_i $;
    \Ensure The processes $ (x^{l,\pi}_{t_i},u^{l,\pi}_{t_i},v^{l,\pi}_{t_i},y^{l,\pi}_{t_i},z^{l,\pi}_{t_i}) $.
    \For { $ l = 0 $ to $ maxstep $}
    \For {$k=0$ to $\kappa + 1$}
    \State $ x_{0}^{l,\pi,m} = a $, $loss_1 = 0$, $loss_2 = 0$;
    \For { $ i = 0 $ to $ N-1 $}
    \State $u^{l,\pi,m}_{t_i} = \varphi_u(t_i, x^{l,\pi,m}_{t_i};\theta_{uv}^l);$
    \State $v^{l,\pi,m}_{t_i} = \varphi_v(t_i, x^{l,\pi,m}_{t_i};\theta_{uv}^l);$
    \State $y^{l,\pi,m}_{t_i} = \varphi_y(t_i, x^{l,\pi,m}_{t_i};\theta_{yz}^l);$
    \State $z^{l,\pi,m}_{t_i} = \varphi_z(t_i, x^{l,\pi,m}_{t_i};\theta_{yz}^l);$
    \State $x^{l,\pi,m}_{t_{i+1}} = x^{l,\pi,m}_{t_i} + u^{l,\pi,m}_{t_i}\Delta t_i + v^{l,\pi,m}_{t_i} \Delta B_{t_i};$
    \State $loss_1 = loss_1 + F(t_i, x_{t_i}^{l,\pi,m}, u_{t_i}^{l,\pi,m}, v_{t_i}^{l,\pi,m}, y_{t_i}^{l,\pi,m}, z_{t_i}^{l,\pi,m}) \Delta t_i;$
    \State $loss_2 = loss_2 + \dfrac{1}{M} \sum\limits_{m=1}^M \left[|F_y(t_i, x_{t_i}^{l,\pi,m}, u_{t_i}^{l,\pi,m}, v_{t_i}^{l,\pi,m}, y_{t_i}^{l,\pi,m}, z_{t_i}^{l,\pi,m})|^2\right];$
    \State $loss_2 = loss_2 + \dfrac{1}{M} \sum\limits_{m=1}^M \left[|F_z(t_i, x_{t_i}^{l,\pi,m}, u_{t_i}^{l,\pi,m}, v_{t_i}^{l,\pi,m}, y_{t_i}^{l,\pi,m}, z_{t_i}^{l,\pi,m})|^2\right];$
    \EndFor
    \State $loss_1 = \dfrac{1}{M} \sum\limits_{m=1}^M \left[loss_1 + \Phi(x_{t_N}^{l,\pi,m})\right]; $
    \If {$k=0$}
    \State $ \theta_{uv}^{l} = Adam(\theta_{uv}^{l}, \nabla loss_1); $
    \State $ \theta^{l} = (\theta_{uv}^{l}, \theta_{yz}^{l}); $
    \Else
    \State $ \theta_{yz}^{l} = Adam(\theta_{yz}^{l}, \nabla loss_2); $
    \State $ \theta^{l} = (\theta_{uv}^{l}, \theta_{yz}^{l}); $
    \EndIf
    \EndFor
    \State $ \theta^{l+1} = (\theta_{uv}^{l}, \theta_{yz}^{l}); $
    \EndFor
    \State $y_{t_i}^{\pi} = \dfrac{1}{M} \sum\limits_{m=1}^M \Big[ \sum\limits_{j=i}^{N-1} -F_x(t_j, x_{t_j}^{l,\pi,m}, u_{t_j}^{l,\pi,m}, v_{t_j}^{l,\pi,m},y_{t_j}^{l,\pi,m}, z_{t_j}^{l,\pi,m}) \Delta t_j - \Phi_x(x_{t_N}^{l,\pi,m}) \Big]. $
  \end{algorithmic}
\end{algorithm}

\subsection{Example 1: a linear Hamiltonian system}

Firstly,
we consider the following linear quadratic Hamiltonian
\begin{equation}\label{Ex2H}
    H(t,x,y,z) = \langle x, y \rangle + \dfrac{1}{4} \langle y, y \rangle + \langle z, z \rangle, \qquad \Phi(x) = \langle \dfrac{1}{2}Qx, x \rangle,
\end{equation}
where $(x,y,z)\in\mathbb{R}^{n+n+n}$ and $Q$ is a given matrix valued in $\mathbb{R}^{n\times n}$,
and the corresponding Hamiltonian system is given as
\begin{equation}\label{Ex2Hsys}
    \left\{
    \begin{array}{l}
        \mathrm{d} x_t = (x_t + \dfrac{1}{2} y_t) \mathrm{d} t + 2z_t \mathrm{d} B_t, \vspace{1ex} \\
        -\mathrm{d}y_t = y_t \mathrm{d} t - z_t \mathrm{d} B_t, \vspace{1ex} \\
        x_0 = a, \qquad y_T = -\Phi_x(x_T)= -Qx_T,
    \end{array}
    \right.
\end{equation}
which is a linear FBSDE and $B$ is a $1$-dimensional standard Brownian motion.
It can be easily check that this linear FBSDE has a unique solution \cite{hu1995solution,wu1999Fully,zhang2015}.

As is known, the linear FBSDE connects with a Riccati equation.
Suppose the solution of FBSDE \eqref{Ex2Hsys} is in the following form:
\begin{align*}
  y_t = -K_tx_t, \qquad z_t = -M_tx_t.
\end{align*}
Combing it with \eqref{Ex2Hsys},
we then obtain a Riccati equation
\begin{equation}\label{Ex2Riccati}
  \begin{cases}
    \dot{K}_t-\frac{1}{2}K_t^2+2K_t=0,\\
    M_t=0, \qquad K_T = Q,
  \end{cases}
\end{equation}
where $ K_t $ is a matrix function and $ \dot{K}_t $ is the derivative of $ K_t $ with respect to $ t $.
And the solution of  \eqref{Ex2Riccati} can be approximated with the ODE45 method in MATLAB (ODE45 in short),
which solves determined ordinary differential equations with the four-order Runge-Kutta method.
In order to show the performance of our novel proposed algorithms,
we take the numerical solution of the Ricatti equation \eqref{Ex2Riccati} with ODE45 as a benchmark.

Now we consider the corresponding optimal control problem,
\begin{equation*}
    \left\{
    \begin{array}{l}
        \mathrm{d} x_t = u_t \mathrm{d} t + v_t \mathrm{d} B_t, \vspace{1ex}\\
        x_0 = a,
    \end{array}
    \right.
\end{equation*}
with cost functional
\begin{equation*}
    J(u(\cdot),v(\cdot)) = \mathbb{E} \Big \{ \int_0^T f(t, x_t, u_t, v_t) + \Phi(x_T) \Big\},
\end{equation*}
where
\begin{align*}
    f(t,x,u,v) &= \max_{y,z} F(t,x,u,v,y,z) \vspace{1ex} \\
    &= |u-x|^2 + \dfrac{1}{4}|v|^2,
\end{align*}
and
\begin{equation*}
    F(t,x,u,v,y,z) = \langle y, u \rangle + \langle z, v \rangle - \langle x, y \rangle - \dfrac{1}{4} \langle y, y \rangle - \langle z, z \rangle.
\end{equation*}

In this example,
we set $a=(1.0,\cdots,1.0)\in\mathbb{R}^{n}, T=0.1$ and give the form of $Q$ as
\begin{equation*}
    \begin{bmatrix}
        1 & \lambda & \lambda & \cdots & \lambda \\
        \lambda & 1 & \lambda & \cdots & \lambda \\
        \lambda & \lambda & 1 & \cdots & \lambda \\
        \vdots & \vdots & \vdots & \ddots & \vdots \\
        \lambda & \lambda & \lambda & \cdots & 1
    \end{bmatrix}
\end{equation*}
where $\lambda$ is a given constant between 0 and 1.
For example,
if $\lambda=0.0$,
$Q=E_n$ is a $ n $-order unit matrix,
The numerical result of \eqref{Ex2Riccati} with ODE45 can be solved with $K_0=1.1573E_n$ for $n=100$,
then the value of $y_0$ can be obtained by
\begin{equation*}
    y_0 = -K_0x_0 = -1.1573a.
\end{equation*}
which is taken as the benchmark results in this example.

Even though the function $f$ can be solved explicitly in this example,
we calculate the numerical results through both the two proposed algorithms and regard that $f$ does not have an explicit form in Algorithm \ref{alg:double}.
The comparison results on the approximated solution $y_0$ between our proposed stochastic control methods (Algorithm \ref{alg:main} and \ref{alg:double}) and the Deep FBSDE method  are shown in \autoref{tab-ex1},
and the solution with ODE45 is regarded as the benchmark.
Note that when the initial state $x=a$ , the solution of $y_0$ is a vector and all its elements are equal,
thus we show the value of the first element of $y_0$  in \autoref{tab-ex1}.
Moreover, we show the relative errors between our approximated solution of $y_0$ and that of ODE45,
and compute the variances among ten independent runs of the approximated solution $y_0$.
We also change the parameter $\lambda$, and study the corresponding approximation results.

\begin{sidewaystable}[htbp]
  \centering
  \renewcommand\arraystretch{1.5}
  \caption{The implementations of different terminal $Q$ with $n=100$}
  \begin{tabular}{|c|c|c|c|c|c|c|c|c|c|c|}
    \hline
    \multirow{2}{*}{} & \multirow{2}{*}{Riccati} & \multicolumn{3}{|c|}{DEEP FBSDE}&  \multicolumn{3}{|c|}{Alg 1}&  \multicolumn{3}{|c|}{Alg 2} \\
    \cline{3-11} & & Mean &Rela. Error & Var.& Mean &Rela. Error &Var.& Mean &Rela. Error &Var.\\
    \hline
    $\lambda=0.0$ & -1.1573 & -1.15733 & 2.907e-05 & 6.209e-08 & -1.15751 & 1.797e-04 & 3.141e-07 & -1.15651 & 6.839e-04 & 3.065e-06 \\
    \hline
    $\lambda=0.2$ & -11.8093 & -11.6314 & 1.506e-02 & 2.392e-01 & -11.8113 & 1.733e-04 & 2.441e-03 & -11.8222 & 1.095e-03 & 5.997e-03 \\
    \hline
    $\lambda=0.4$ & -15.2711 & -14.1265 & 7.495e-02 & 9.101e-01 & -15.3120 & 2.681e-03 & 1.739e-03 & -15.2264 & 2.930e-03 & 1.622e-01 \\
    \hline
    $\lambda=0.6$ & -16.9860 & \textbf{-10.4284} & 3.861e-01 & 2.753e+01 & -17.0461 & 3.859e-03 & 3.563e-03 & -17.0516 & 3.860e-03 & 4.633e-02 \\
    \hline
    $\lambda=0.8$ & -18.0101 & \textbf{-9.2087} & 4.887e-01 & 2.279e+01 & -17.9844 & 1.426e-03 & 2.528e-02 & -18.1494 & 7.733e-03 & 2.726e-02 \\
    \hline
    $\lambda=1.0$ & -18.6920 & \textbf{-9.9503} & 4.677e-01 & 2.300e+01 & -18.7162 & 1.297e-03 & 2.000e-02 & -18.8837 & 1.026e-02 & 8.245e-02 \\
    \hline
  \end{tabular}
  \label{tab-ex1}
\end{sidewaystable}

From \autoref{tab-ex1},
we can see that the novel proposed algorithms show much more stable performance than the Deep FBSDE method.
For different terminals with different parameters $\lambda$,
the novel proposed algorithms demonstrate more stable relative errors and variances.
The Deep FBSDE method perform well when the terminal is a unit matrix ($\lambda=0.0$),
but when we change the terminal to other forms ($\lambda\not=0.0$),
the results of the Deep FBSDE method diverge. As we know, the learning rate is one of the important factors affecting the approximation results. When we choose smaller learning rate,
the Deep FBSDE method can converge, but it need much more iteration steps than our proposed algorithms. As an example for the case $\lambda\not=0.0$ with smaller learning rate, we show the approximation results in the following Figure \ref{figEx2-2} for $\lambda=0.8$.

In Figure \ref{figEx2}, we show the curves and the variances of the approximated results $y_0$ with different iteration steps for $\lambda=0.0$ and $\lambda=0.8$ respectively,
and the black lines represent the results with ODE45 which is taken as the benchmark. The upper two figures in Figure \ref{figEx2} exhibit the results for $\lambda=0.0$.
From the upper left figure,
we can see that when the number of iteration steps is close to 10000,
the approximated solution $y_0$ with Algorithm \ref{alg:main}, \ref{alg:double} and the Deep FBSDE method are all very close to the results with ODE45.
Moreover, that of Algorithm \ref{alg:main} and \ref{alg:double} have smaller variation scopes among ten independent runs and converge within less iteration steps than that of the Deep FBSDE method.
The upper right figure shows that when the number of iteration steps tends to be 10000,
the variance curve of $y_0$ with Algorithm \ref{alg:main} and \ref{alg:double} are also close to that of the Deep FBSDE method.
The lower two figures in Figure \ref{figEx2} exhibit the results for $\lambda=0.8$.
We can see that when the number of iteration steps tends to be 10000, the approximation results of Algorithm \ref{alg:main} and \ref{alg:double} are close to the benchmark. However, that of the Deep FBSDE method is far from the benchmark,
and the variation scope and variance increase with the increase of iteration steps.

\begin{figure}[H]
    \centering
    \includegraphics[scale=0.45]{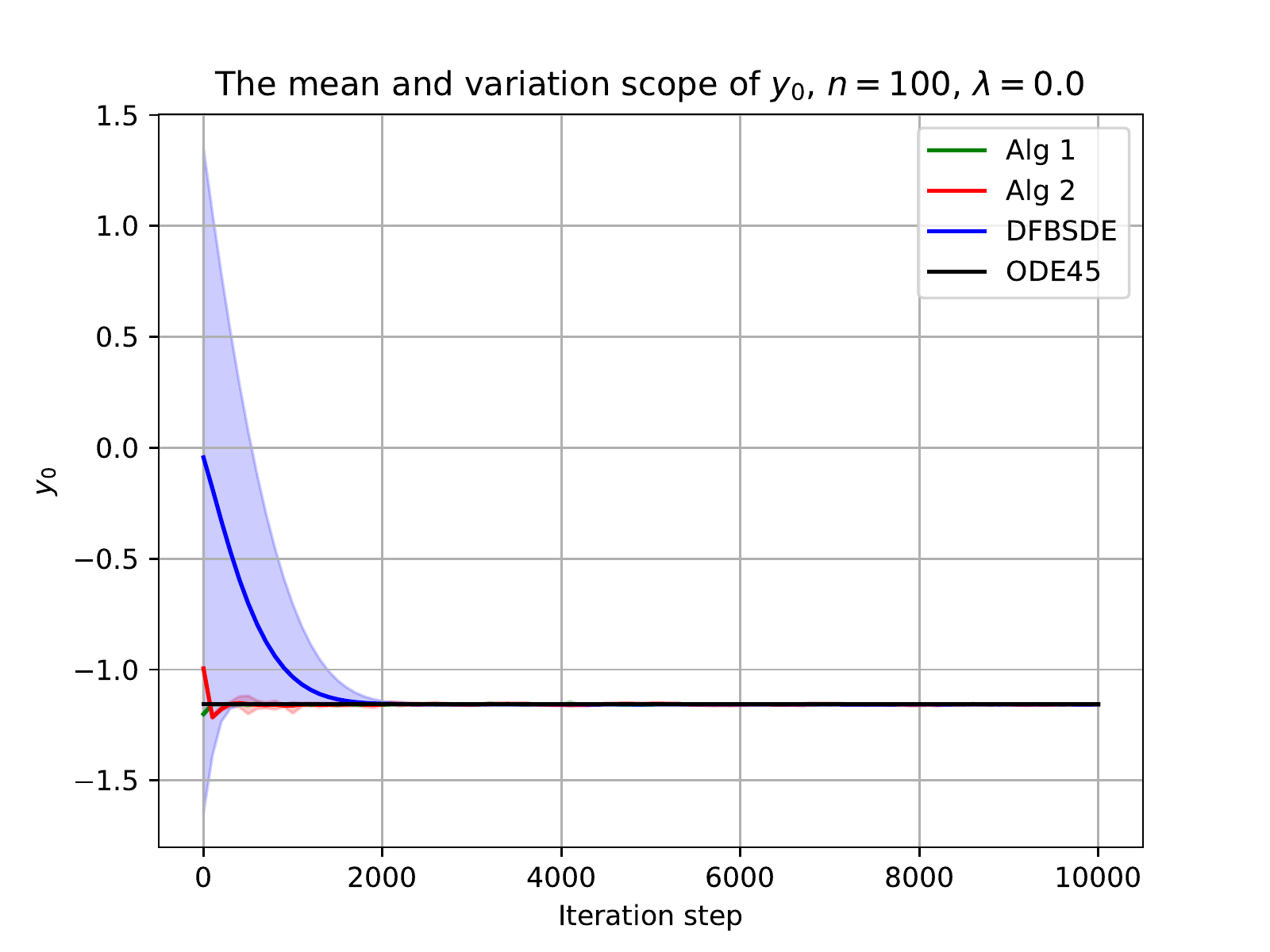}
    \includegraphics[scale=0.45]{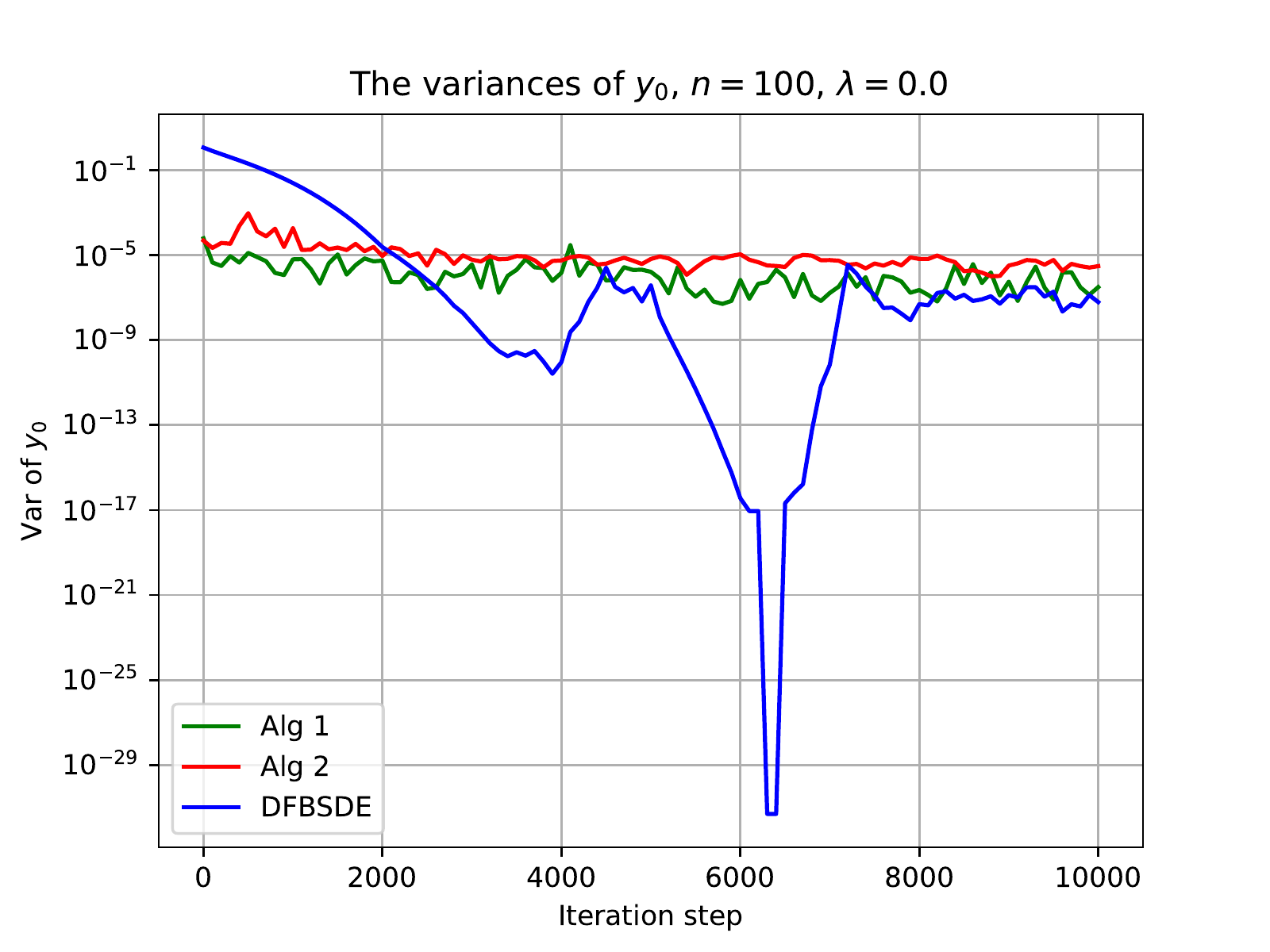}
    \includegraphics[scale=0.45]{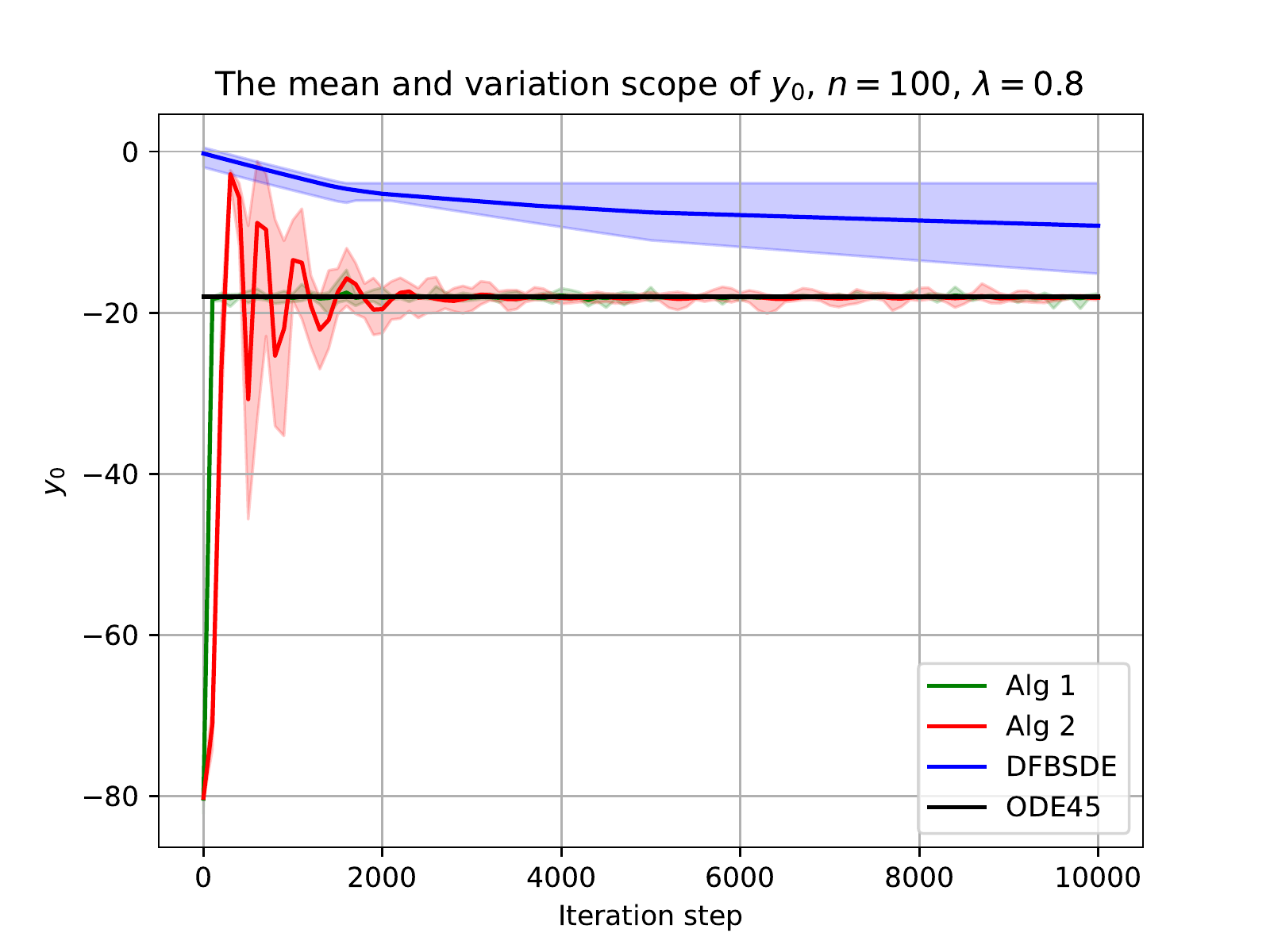}
    \includegraphics[scale=0.45]{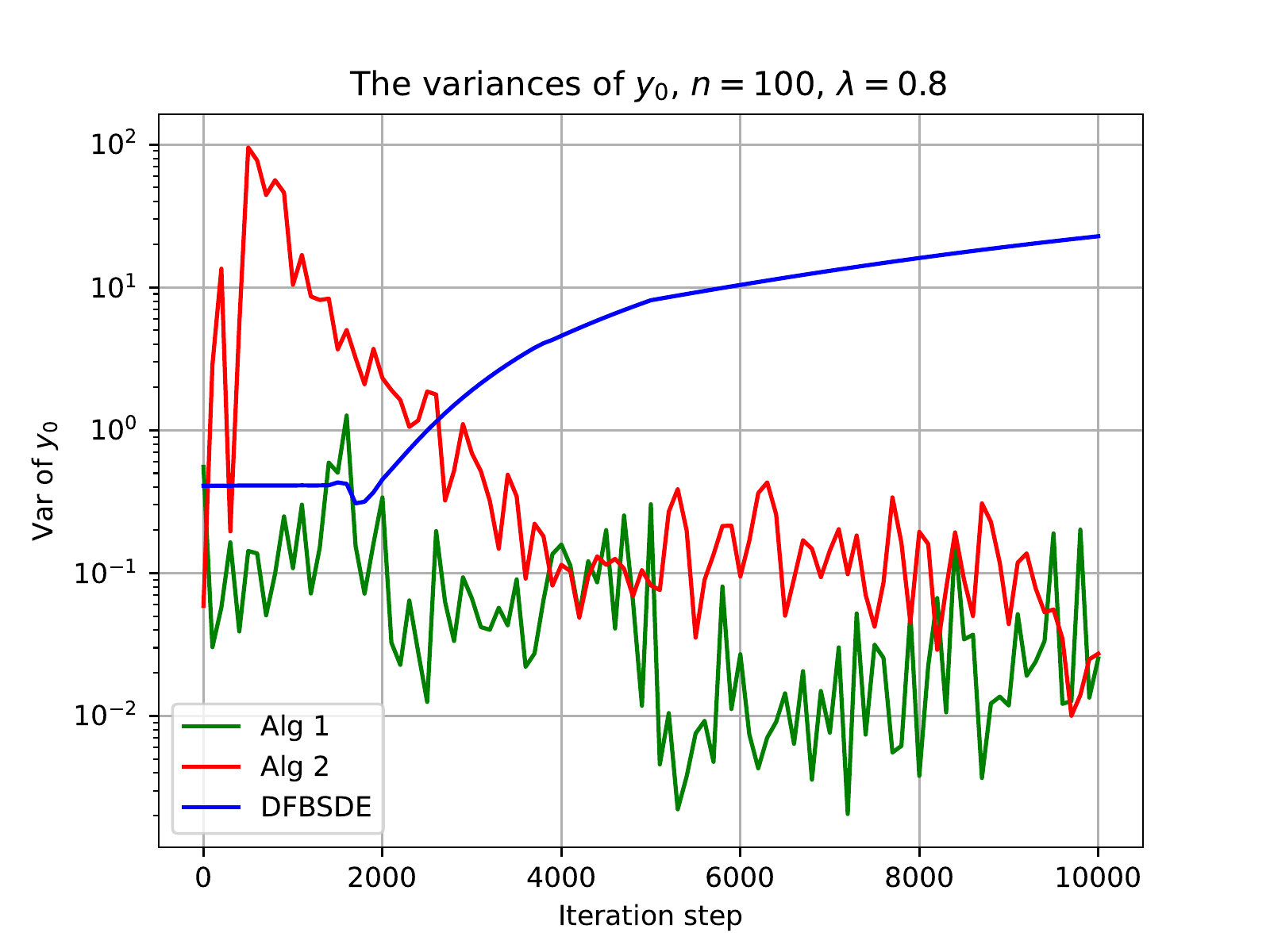}
    \caption{Approximation results with a piecewise decay learning rate from $3\times 10^{-3}$ to $1\times10^{-3}$ for $\lambda=0.0$ in the upper figures and $\lambda=0.8$ in the lower figures. The left figures show the mean and variation scopes of the approximated solution $y_0$ among 10 independent runs, and the right figures exhibit the variance curves of $y_0$ among 10 independent runs. The black lines in the left figure represent the results with ODE45 which are taken as the benchmarks. We can see that our novel proposed algorithms(Algorithms 1 and 2) have more stable convergence and converge within less iteration steps.
    }
    \label{figEx2}
\end{figure}

\begin{figure}[H]
    \centering
    \includegraphics[scale=0.45]{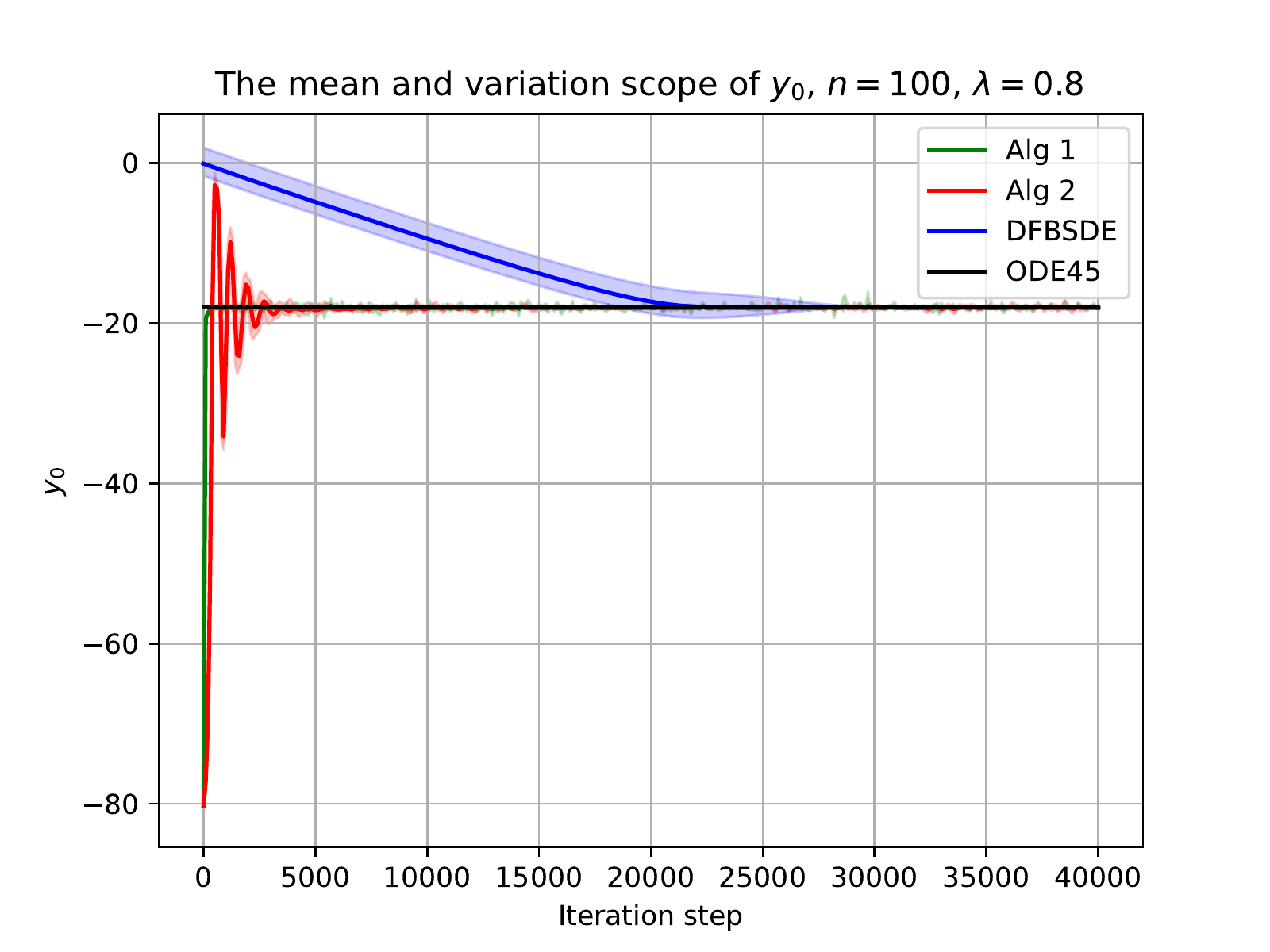}
    \includegraphics[scale=0.45]{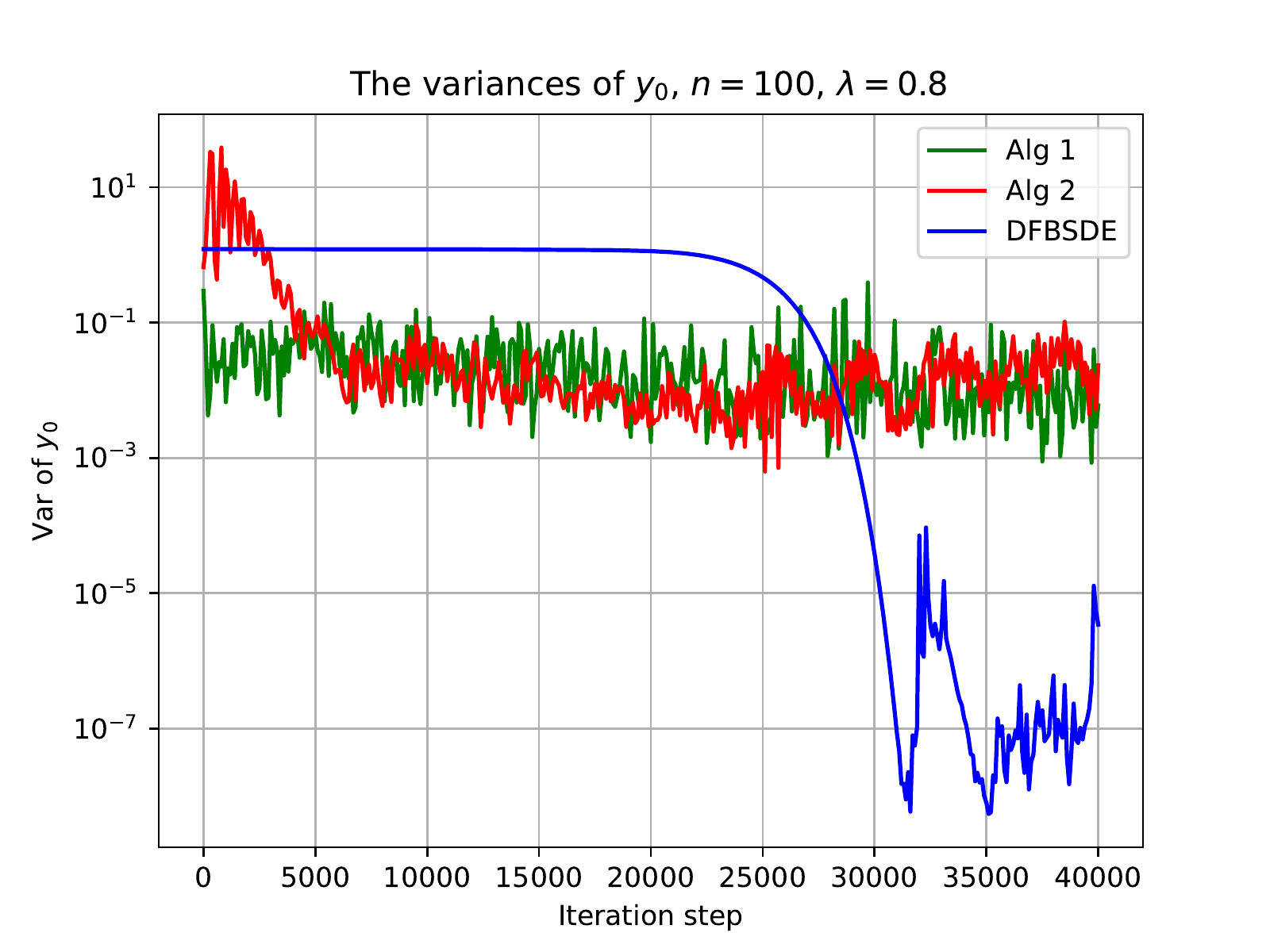}
    \caption{Approximation results with a constant learning rate of $1\times 10^{-3}$. We can see from the left figure that comparing with the novel proposed algorithms(Algorithms 1 and 2), the Deep FBSDE method needs more iteration steps to achieve stable convergence, and it has a much smaller variance at the end of the training from the right figure.
    }
    \label{figEx2-2}
\end{figure}

\subsection{Example 2: a nonlinear Hamiltonian system}\label{ssec-ex3}

Given the Hamiltonian $H$ as
\begin{equation}\label{Ex3H}
    H(t,x,y,z) = \dfrac{1}{2} \langle y, y\circ\cos ^2 x \rangle + \dfrac{1}{2} \langle z, z\circ\sin ^2 x \rangle + \langle y, \cos x \rangle + \langle z, \sin x \rangle - \dfrac{1}{2} \langle x,x \rangle
\end{equation}
and
\begin{equation}\label{Ex3Phi}
    \Phi(x) = \dfrac{1}{2}\langle x, x \rangle,
\end{equation}
where $(x,y,z)\in\mathbb{R}^{n+n+n}$.
Here
$\circ$ represents the Hadamard product,
\[
    x \circ y = (x_1y_1, x_2y_2, \cdots, x_ny_n) \in \mathbb{R}^n, \qquad \forall x, y\in\mathbb{R}^n
\]


The corresponding Hamiltonian system is
\begin{equation}\label{Ex3Hsys}
\left\{
\begin{array}{l}
\mathrm{d}x_t = \cos x_t\circ (y_t\circ\cos x_t +1) \mathrm{d}t + \sin x_t\circ(z_t \circ\sin x_t +1) \circ \mathrm{d} B_t, \vspace{1ex}\\
-\mathrm{d}y_t = \Big[ -y_t\circ\sin x_t\circ (y_t\circ \cos x_t+1) + z_t\circ \cos x_t \circ(z_t\circ \sin x_t+1) - x_t \Big] \mathrm{d} t \\
\qquad \qquad - z_t\circ \mathrm{d} B_t, \vspace{1ex}\\
x_0 = a, \qquad y_T=-\Phi_x(x_T),
\end{array}
\right.
\end{equation}
where $B$ is a $n$-dimensional Brownian motion.

Here we introduce a stochastic optimal control problem which is different from \eqref{eq:sys_two_net_linear}:
\begin{equation}\label{Ex3StateX}
    \left\{
    \begin{array}{l}
        \mathrm{d} x_t = \cos x_t\circ(u_t+1) \mathrm{d} t + \sin x_t\circ(v_t+1)\circ \mathrm{d} B_t \vspace{1ex} \\
        x_0 = a,
    \end{array}
    \right.
\end{equation}
with the cost functional
\begin{equation*}
J(u(\cdot),v(\cdot))=\mathbb{E} \left\lbrace \int_{0}^{T} f(t, x_t, u_t, v_t) dt + \Phi(x_T) \right\rbrace,
\end{equation*}
where $f(t, x, u, v)$ is given as
\begin{align*}
    f(t, x, u, v) & = \max_{y,z} F(t,x,u,v,y,z) \vspace{1ex} \\
    & = \max_{y,z} \left\{ \langle y, \cos x\circ(u+1)\rangle + \langle z,\sin x\circ(v+1)\rangle - H(t,x,y,z) \right\}.
\end{align*}
Then $f(t, x, u, v)$ can be solved as
\begin{equation*}
    f(t, x, u, v) = \dfrac{1}{2} \big[ \langle x,x \rangle + \langle u,u \rangle + \langle v,v \rangle \big],
\end{equation*}
and
\begin{equation*}
    \begin{aligned}
        y &= u \circ \dfrac{1}{\cos x}, \\
        z &= v \circ \dfrac{1}{\sin x}.
    \end{aligned}
\end{equation*}

Define $h(t,x,y,v,y,z)$ as
\begin{equation*}
\begin{aligned}
    h(t,x,u,v,y,z) = \langle y, \cos x\circ(u+1) \rangle + \langle z, \sin x\circ(v+1) \rangle - f(t,x,u,v),
\end{aligned}
\end{equation*}
then we get the value of $y_0$ by

\begin{align*}
    y_0 &= \mathbb{E} \left\lbrace \int_0^{T} h_x(t,x_t,u_t,v_t,y_t,z_t)dt - \Phi_x(x_T) \right\rbrace\\
    & = \mathbb{E} \left\lbrace \int_{0}^{T} (-u_t\circ \tan x_t\circ(u_t +1) + v_t \circ\cot x_t\circ(v_t+1) - x_t) - \Phi_x(x_T) \right\rbrace.
\end{align*}

We set $a=(1.0,\cdots,1.0)\in\mathbb{R}^{n}$ and $T=0.1$.
The comparison results among Algorithm \ref{alg:main}, Algorithm \ref{alg:double} and the Deep FBSDE method are shown in Figure~\ref{figEx3}.
We can see from the left figure that when the number of iteration steps is 5000, The approximated value of $Y_0$ with  Algorithm \ref{alg:main} and the Deep FBSDE method are very close, which are $-1.0835$ and $-1.0834$, respectively.
And the approximated solution with Algorithm \ref{alg:double} is $-1.1208$, which is slightly smaller than that with Algorithm \ref{alg:main} and the Deep FBSDE method. Similar to the first example, the variation scope of $y_0$ for Algorithm \ref{alg:main}, \ref{alg:double} are much smaller than that of the Deep FBSDE method. Besides, at the end of the training, the variances of $y_0$ for Algorithm \ref{alg:main}, \ref{alg:double} and the Deep FBSDE method are very close.

\begin{figure}[H]
    \centering
    \includegraphics[scale=0.45]{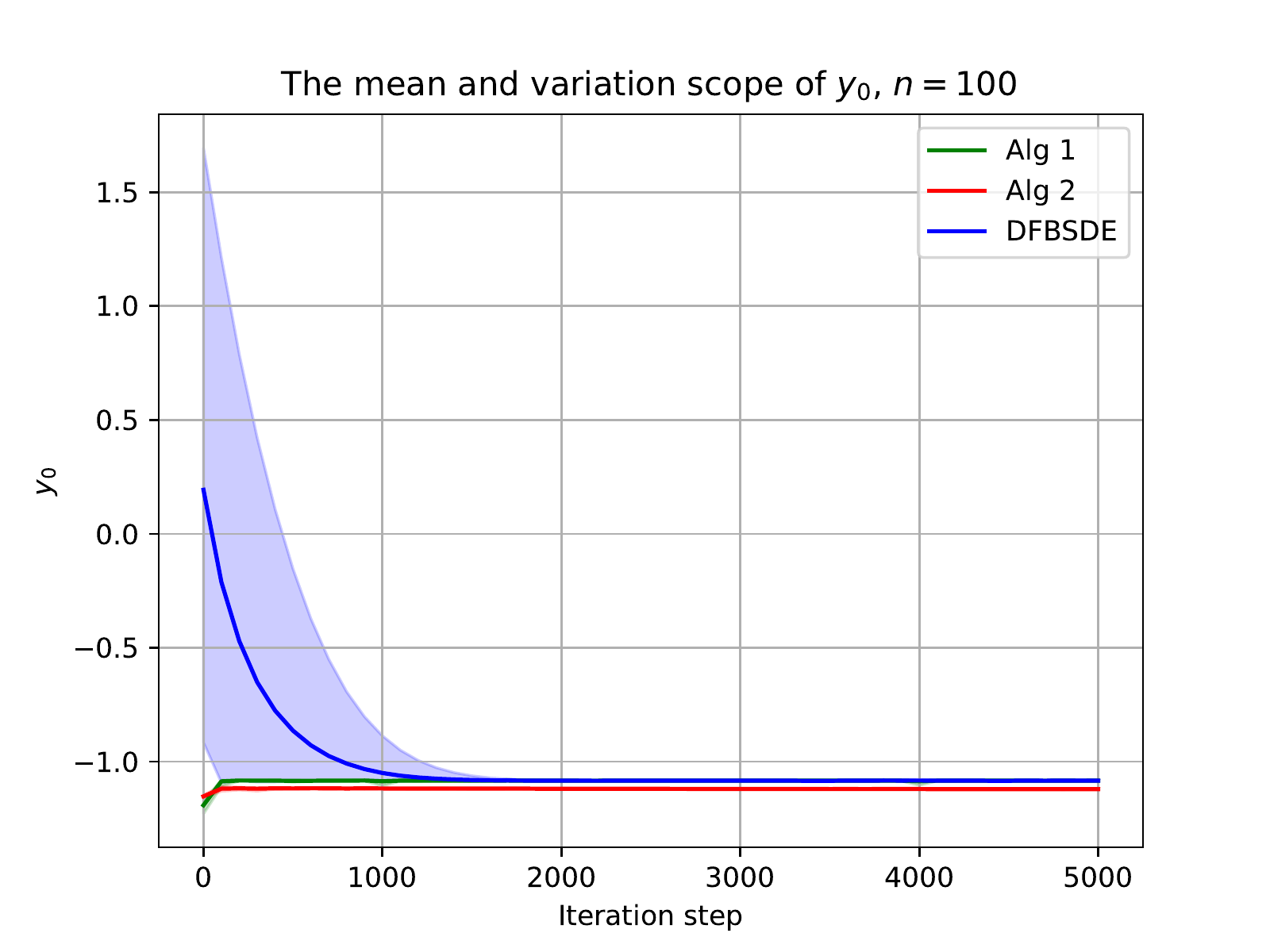}
    \includegraphics[scale=0.45]{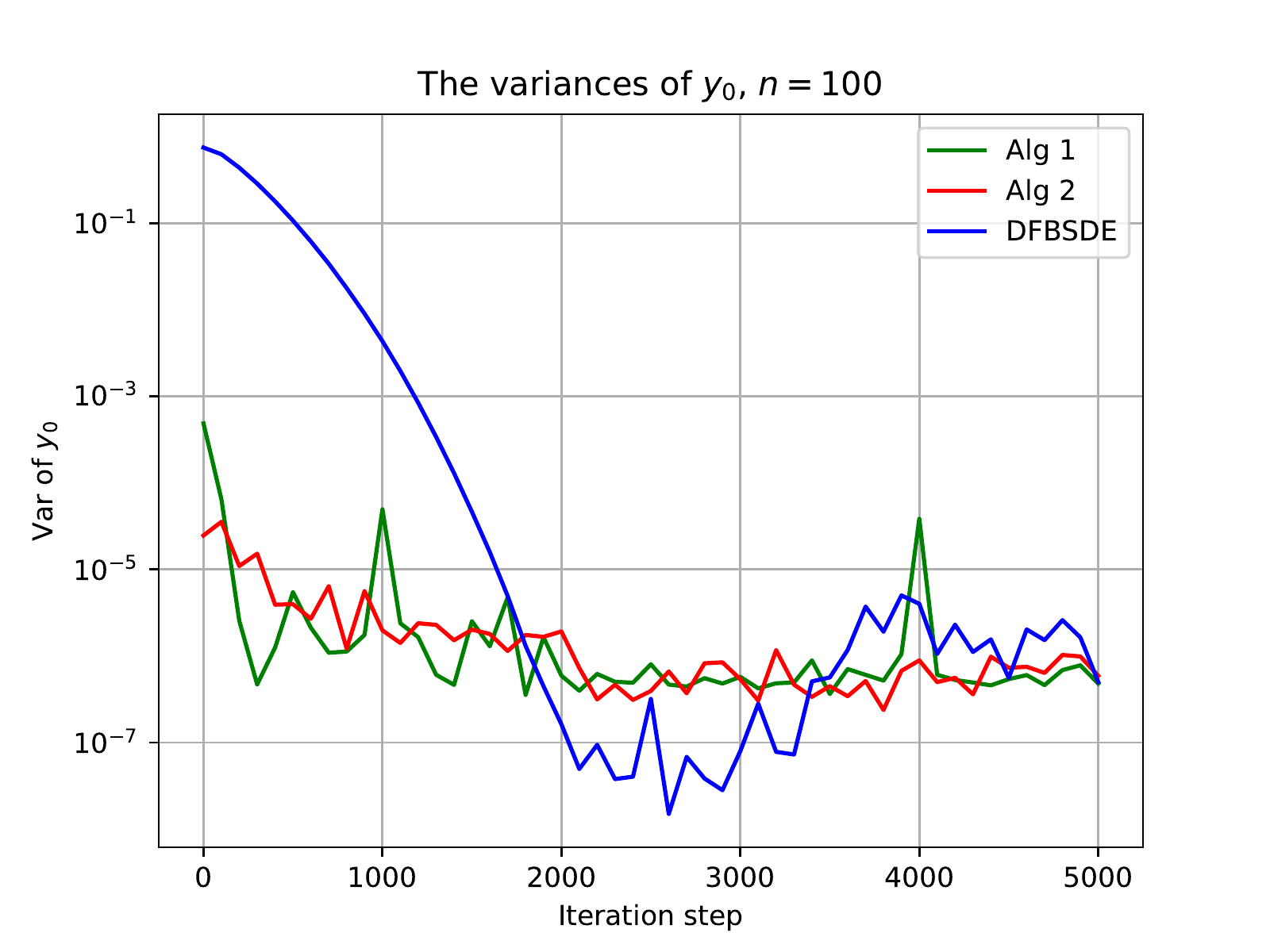}
    \caption{We can see from the figure that  when the number of iteration steps tends to be 5000, the approximated values of $y_0$ for all the three methods are close. However, the Deep FBSDE method shows larger variation scopes than Algorithm \ref{alg:main} and \ref{alg:double} during the training. And the variances of these three methods are also close  at the end of the training.}
    \label{figEx3}
\end{figure}

\subsection{Example 3: a Hamiltonian system with exponents in the drift term}\label{ssec-ex4}

In the third example, we solve a Hamiltonian system with exponents in the drift term. Consider the following Hamiltonian
\begin{equation}\label{eq-Hamil-exp}
    H(t,x,y,z) = \log \left(\sum_{i=1}^n\exp(y_i)\right) + \dfrac{1}{2} \langle y, y \rangle + \langle z, z \rangle + \langle z, x \rangle + \dfrac{1}{5} \langle x, x \rangle,
\end{equation}
where $(x,y,z)\in\mathbb{R}^{n+n+n}$ and $y = (y_1,\cdots,y_n)$.
The terminal function is  given as $\Phi(x)=\dfrac{1}{2}\langle x, x \rangle$.
Then the Hamiltonian system we need to solve is given as following,
\begin{equation}\label{Ex4Hsys}
\left\{
\begin{array}{l}
\mathrm{d}x_t = \left[\exp(y_t) \left(\sum_{i=1}^n\exp(y_{it})\right) ^{-1} + y_t \right] \mathrm{d}t + (x_t + 2z_t) \circ \mathrm{d} B_t, \vspace{1ex}\\
-\mathrm{d}y_t = (\dfrac{2}{5} x_t + z_t) \mathrm{d} t - z_t \circ \mathrm{d} B_t, \vspace{1ex}\\
x_0 = a, \qquad y_T=-\Phi_x(x_T),
\end{array}
\right.
\end{equation}
where $\exp(y_t) = (\exp(y_{it}),\cdots,\exp(y_{nt}))$ and $B$ is a $n$-dimensional Brownian motion.

The corresponding stochastic optimal control problem is
\begin{equation*}\label{Ex4Con}
    \begin{cases}
    \mathrm{d} x_t = u_t\mathrm{d}t + v_t\circ\mathrm{d}B_t, \vspace{1ex} \\
    x_0 = a,
    \end{cases}
\end{equation*}
with the cost functional
\begin{equation*}
\begin{aligned}
    J(u(\cdot),v(\cdot))=\mathbb{E} \left\lbrace \int_{0}^{T} f(t, x_t, u_t, v_t) \mathrm{d} t + \Phi(x_T) \right\rbrace,
\end{aligned}
\end{equation*}
where
\begin{equation}\label{Ex4f}
\begin{aligned}
    f(t,x,u,v) &= \max_{y,z} F(t,x,u,v,y,z), \\
    &= \max_{y,z} \left\{\langle y, u\rangle + \langle z,v\rangle - H(t,x,y,z) \right\}.
\end{aligned}
\end{equation}
Different from the previous two examples,
in this example, the function $f$ defined as \eqref{Ex4f} does not have an explicit representation. In this situation, Algorithm \ref{alg:main} is not applicable, thus we mainly make the comparison between the results of Algorithm \ref{alg:double} and the Deep FBSDE method.

We set $T=0.2$, $a=(0.5,\cdots,0.5)\in\mathbb{R}^{n}$ and $n=100$ in this example.
The comparison results between Algorithm \ref{alg:double} and the Deep FBSDE method are shown in Figure \ref{figEx4}. We can see that at the end of the training, the approximated solution $y_0$ of the two methods are close, and both the variances are small enough. Similar with the previous two examples,  Algorithm \ref{alg:double} shows smaller variation scope and converges within less iteration steps than the Deep FBSDE method.

\begin{figure}[H]
    \centering
    \includegraphics[scale=0.45]{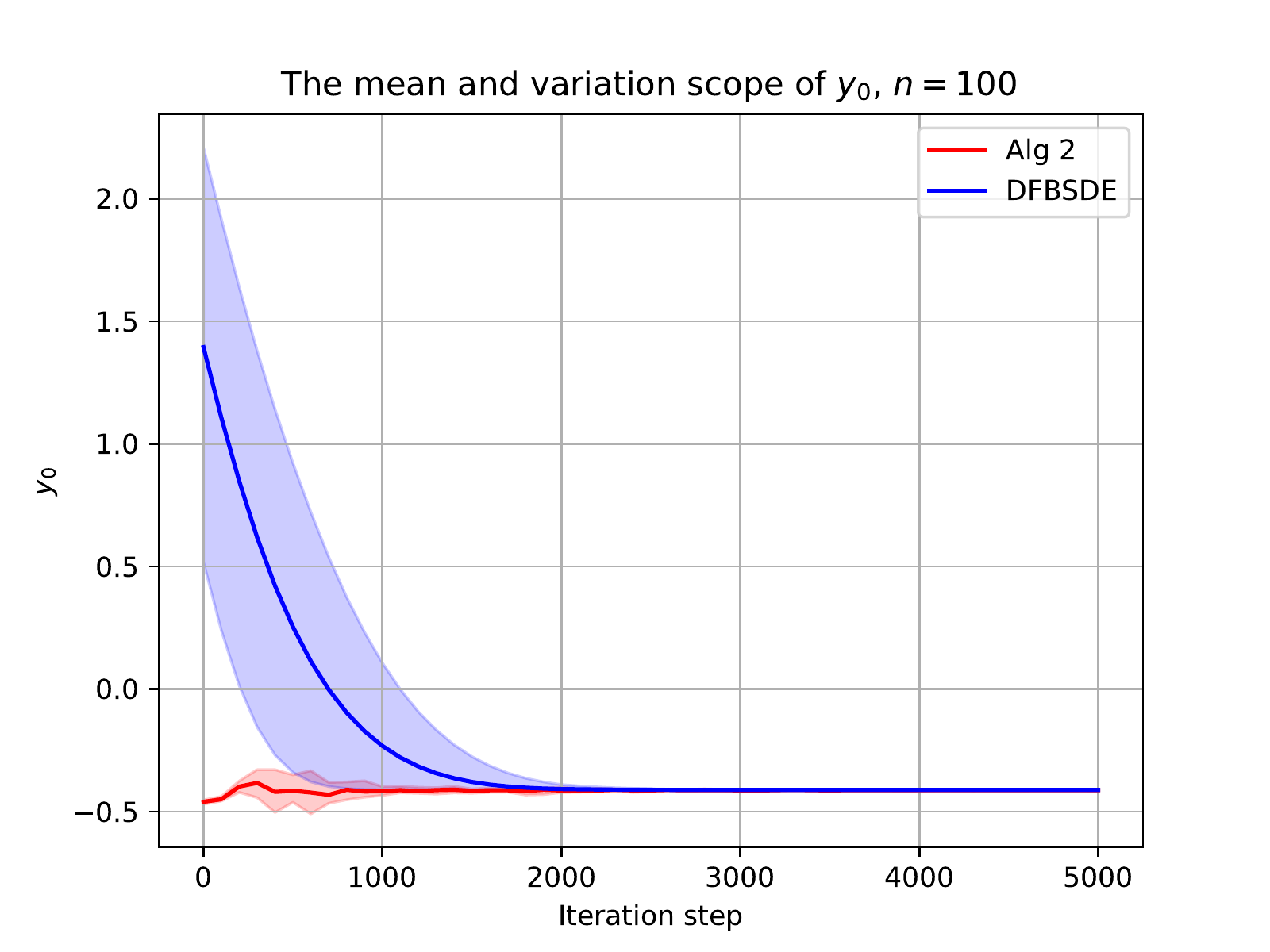}
    \includegraphics[scale=0.45]{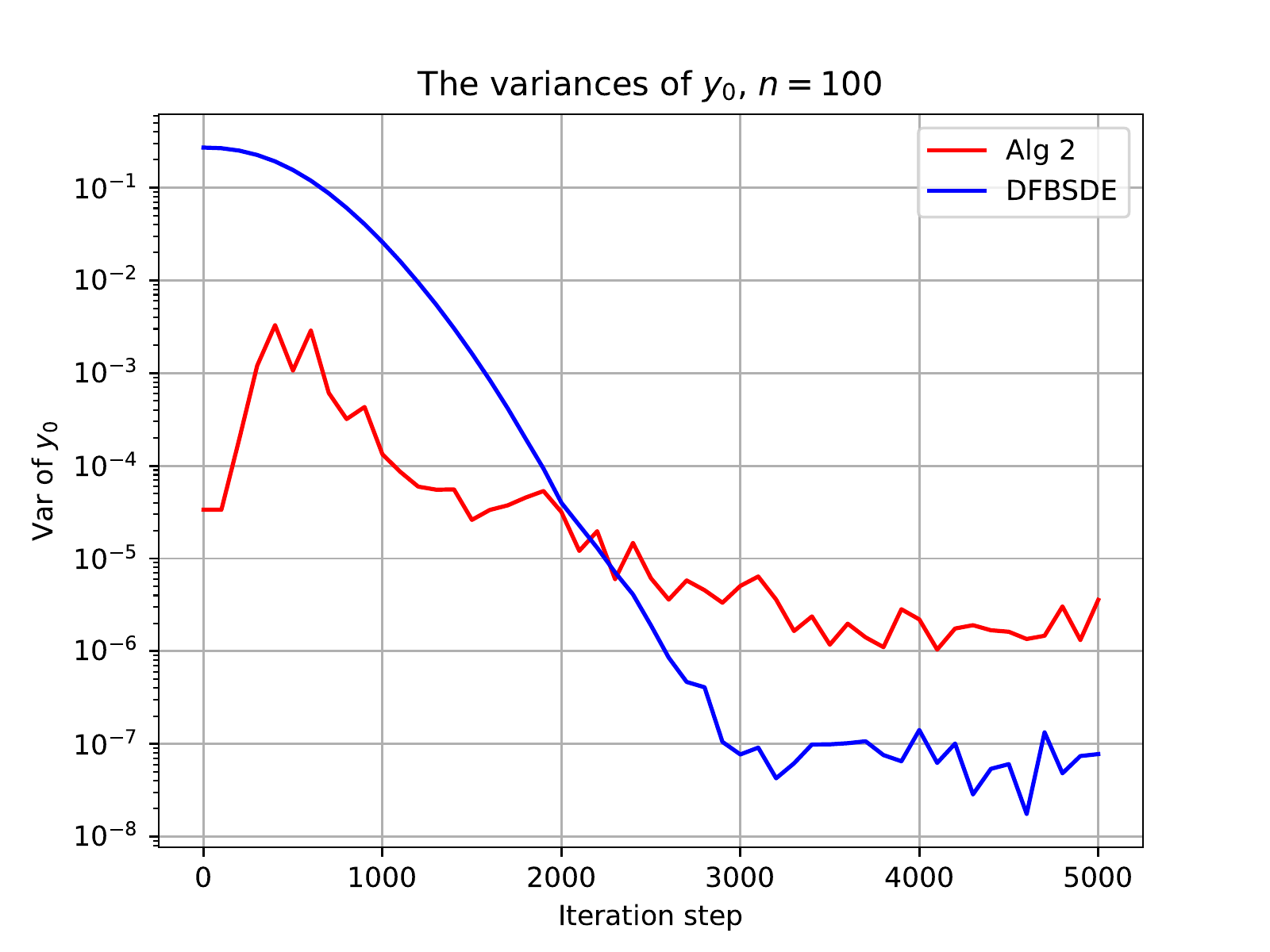}
    \caption{We can see from the above figure that at the end of training, the approximated solutions of $y_0$ with both Algorithm \ref{alg:double}  and the Deep FBSDE method are very close, and the mean values of $y_0$ among 10 independent runs are $-0.41297$ for Algorithm \ref{alg:double} and $-0.41211$ for the Deep FBSDE method.
    The variance of Algorithm \ref{alg:double} is slightly larger than that of the Deep FBSDE method,
    but it converges within less iteration steps.}
    \label{figEx4}
\end{figure}

In the implementations of Algorithm \ref{alg:double} for all the three examples, we minimize the norm of the derivatives of the function $F$ with respect to $y$ and $z$ according to \eqref{eq-loss-2} so that they are equal to $0$. As an alternative, we can also maximize the cost functional defined as \eqref{eq-J-F} in the implementations.

\section{Conclusion}\label{sec-con}

In this paper,
different from the general way of solving the FBSDEs,
we propose a novel method which solve the Hamiltonian system from the view of the stochastic optimal control via deep learning. Two different algorithms suitable for different cases are developed. From the numerical results, the novel proposed Algorithms \ref{alg:main} and \ref{alg:double} demonstrate faster convergence rate  and more stable performance than the Deep FBSDE method. In some cases, Algorithm \ref{alg:main} and \ref{alg:double} shows higher accuracy.

\bibliographystyle{ieeetr}
\bibliography{ref}

\end{document}